\documentclass[graybox]{svmult}

\usepackage{mathptmx}       
\usepackage{helvet}         
\usepackage{courier}        
\usepackage{type1cm}        
\usepackage{makeidx}         
\usepackage{graphicx}        
\usepackage{multicol}        
\usepackage[bottom]{footmisc}

\usepackage{amsmath,amssymb}
\usepackage{calrsfs}
\usepackage{afterpage} 
\usepackage{lscape}
\usepackage{url}

\newcommand{\F}{{\mathbb{F}}}
\newcommand{\Z}{{\mathbb{Z}}}
\newcommand{\Q}{{\mathbb{Q}}}
\newcommand{\N}{{\mathbb{N}}}
\newcommand{\R}{{\mathbb{R}}}
\newcommand{\C}{{\mathbb{C}}}
\newcommand{\ba}{{\mathbf{a}}}
\newcommand{\bc}{{\mathbf{c}}}

\newcommand{\cE}{{\mathcal{E}}}
\newcommand{\cH}{{\mathcal{H}}}
\newcommand{\cO}{{\mathcal{O}}}
\newcommand{\fX}{{\mathfrak{X}}}
\newcommand{\Irr}{{\operatorname{Irr}}}
\newcommand{\Hom}{{\operatorname{Hom}}}

\renewcommand{\leq}{\leqslant}
\renewcommand{\geq}{\geqslant}

\newcommand{\GAP}{{\sf GAP}}
\newcommand{\MA}{{\sf MeatAxe}}
\newcommand{\IMA}{{\sf IntegralMeatAxe}}
\newcommand{\GL}{{\operatorname{GL}}}

\spnewtheorem{thm}{Theorem}[section]{\bfseries}{\itshape}
\spnewtheorem{prop}[thm]{Proposition}{\bfseries}{\itshape}
\spnewtheorem{lem}[thm]{Lemma}{\bfseries}{\itshape}
\spnewtheorem{cor}[thm]{Corollary}{\bfseries}{\itshape}
\spnewtheorem{conj}[thm]{Conjecture}{\bfseries}{\itshape}
\spnewtheorem{abs}[thm]{\nocaption}{\bfseries}{\rmfamily}

\spnewtheorem{defn}[thm]{Definition}{\bfseries}{\rmfamily}
\spnewtheorem{exmp}[thm]{Example}{\itshape}{\rmfamily}
\spnewtheorem{rem}[thm]{Remark}{\itshape}{\rmfamily}

\spnewtheorem*{sub}{\nocaption}{\bfseries}{\rmfamily}
\newcommand{\subT}[1]{\begin{sub}\hspace*{-0.5em}{\bf #1.}}
\newcommand{\absT}[1]{\begin{abs}{\bf #1.}}
\spnewtheorem{obs}{Observation}{\bfseries}{\itshape}

\makeindex            
\smartqed

\begin{document}

\title*{
Invariant bilinear forms on \\
$W$-graph representations and \\
linear algebra over integral domains}
\titlerunning{Invariant bilinear forms} 
\author{Meinolf Geck and J\"urgen M\"uller}
\authorrunning{Meinolf Geck and J\"urgen M\"uller} 
\institute{Meinolf Geck \at IAZ-Lehrstuhl f\"ur Algebra, Universit\"at
Stuttgart, Pfaffenwaldring 57, 70569 Stuttgart, Germany
\email{meinolf.geck@mathematik.uni-stuttgart.de}
\and J\"urgen M\"uller \at Arbeitsgruppe Algebra und Zahlentheorie,
Bergische Universit\"at Wuppertal, Gau{\ss}-Stra{\ss}e 20,
42119 Wuppertal, Germany
\email{juergen.mueller@math.uni-wuppertal.de}}

\maketitle

\abstract{Lie-theoretic structures of type $E_8$ (e.g., Lie groups and
algebras, Iwahori--Hecke algebras and Kazhdan--Lusztig cells, $\ldots$)
are considered to serve as a ``gold standard'' when it comes to judging
the effectiveness of a general algorithm for solving a computational
problem in this area. 
Here, we address a problem that occurred in our previous work on 
decomposition numbers of Iwahori--Hecke algebras, namely, the computation 
of invariant bilinear forms on so-called $W$-graph representations. We 
present a new algorithmic solution which makes it possible to produce 
and effectively use the main results in further applications.}

\section{Introduction} \label{sec0}

This paper is concerned with the representation theory of Iwahori--Hecke
algebras. Such an algebra $\cH$ is a certain deformation of the group 
algebra of a finite Coxeter group $W$. In \cite{myedin}, the notion
of ``balanced representations'' of $\cH$ was introduced, which has turned 
out to be useful in several applications. We mention here the construction 
of cellular structures on $\cH$ (see, e.g., \cite[Chap.~2]{geja}), the 
determination of decomposition numbers of $\cH$ (see \cite{gemu}), and the 
computation of Lusztig's function $\ba\colon W \rightarrow \Z$ (see 
\cite[\S 4]{geha}). To check whether a given 
representation of $\cH$ is balanced or not is a computationally hard problem; 
it involves the construction of a certain invariant bilinear form on the 
underlying $\cH$-module. It has been conjectured in \cite{myedin} that 
so-called ``$W$-graph representations'' of $\cH$ are always balanced. But 
even if such a theoretical result were known to be true, certain applications 
(e.g., the determination of decomposition numbers) would still require the 
explicit knowledge of the Gram matrices of the invariant bilinear forms. 
In this paper, we discuss algorithms for the construction of these Gram 
matrices for $W$ of exceptional type. The biggest challenge---by far---is 
the case where $W$ is of type $E_8$. (The distinguished role of $E_8$ when
it comes to performing explicit computations is highlighted in various 
recent survey articles; see, e.g., Garibaldi \cite{gari}, Lusztig 
\cite{shaw}, Vogan \cite{VoE8}). 

In the situations of interest to us, the algebra $\cH$ is defined over 
the field of rational functions $K=\Q(v)$ (where $v$ is an indeterminate); 
it has a natural basis $\{T_w\mid w\in W\}$. Explicit models for the
irreducible representations of $\cH$ are known by the work of Naruse
\cite{Naruse0}, Howlett and Yin  \cite{How}, \cite{HowYin}. Now let us fix
an irreducible matrix representation $\fX\colon \cH\rightarrow K^{d\times d}$.
In order to show that $\fX$ is balanced, one needs to determine a non-zero 
symmetric matrix $P \in K^{d\times d}$ such that 
\[P\,\fX(T_w)=\fX(T_{w^{-1}})^{\operatorname{tr}} \,P\qquad\mbox{for 
all $w \in W$};\]
this matrix $P$ then has to satisfy certain additional properties. Thus,
the computation of $P$ essentially amounts to solving a system of
linear equations; for theoretical reasons, we know that this system has
a unique solution up to multiplication by a scalar. Rescaling a given solution
by a suitable non-zero polynomial in $\Q[v]$, we can assume that all entries 
of $P$ are in $\Z[v]$ and that their greatest common divisor is~$\pm 1$; 
then $P$ is unique up to sign
and is called a ``primitive Gram matrix''. 
The general theory also shows that a particular solution is given by 
\[P_0=\sum_{w \in W} \fX(T_w)^{\text{tr}}\, \fX(T_w) \in K^{d\times d}.\]
Thus, if the matrices $\fX(T_w)$ ($w\in W$) are known and if $|W|$ is not
too large, then we can simply perform the above summation and obtain $P_0$;
rescaling $P_0$ yields a primitive Gram matrix $P$. 
This procedure works for types $F_4$, $E_6$, for example. 

\medskip
Already for type $E_7$, one needs to use a more sophisticated approach as 
described in \cite[\S 4.3]{gemu},
based on Parker's ``standard basis algorithm'' \cite{parker1}, 
in combination with interpolation and modular techniques. 
This also works for type $E_8$, but it is efficient only for 
irreducible representations of dimension up to about $2500$. In our 
previous work on decomposition numbers, this was sufficient to obtain the 
desired results for type $E_8$; see \cite[Remark~4.10]{gemu}. In principle, 
one could have run the above procedure on all irreducible representations of 
type $E_8$, but experiments showed that this would have needed a total of 
nearly one year of CPU time. On the other hand, from a strictly logical 
point of view, one does not need to know exactly how the Gram matrices 
have been obtained, because as an independent verification one can simply 
check that they form a solution to the above system of linear equations. 
However, to store the various primitive Gram matrices requires about 
$28$~GB of disk space, and even the verification alone is a major task as 
it involves the computation of products of (large) matrices with polynomial
entries.~---~In any case, this raises a serious issue of making sure
that our results are reliable and reproducible. 

In our view, the solution to deal with this issue is to develop better 
mathematical tools which make it possible to reproduce the results 
efficiently as needed, and this is what we will do in this paper. Indeed, 
for example, in order to deal with the irreducible representation of 
largest dimension for type $E_8$ (which is $7168$), the old approach would
have needed roughly seven weeks of CPU time, while the one described here 
requires only about $20$ hours, which amounts to a factor of almost $60$. 
(See Section \ref{timings} for more details.)  
In view of the complexity of the task, and the experiences made elsewhere 
with explicit computations in type $E_8$ (see the references cited above), 
it was clear that developing efficient methods would not be a standard, let 
alone press-button application of existing tools from computer algebra.
Maier et al.\ \cite{mllt} proposed an approach based on parallel 
techniques, but type $E_8$ still seems to be a major challenge there. 
Hence one of the purposes of this paper is to give a systematic 
description of the (serial) methods we have used for the computation of 
Gram matrices of invariant bilinear forms for Iwahori--Hecke algebras.

The basic strategy in our approach is to reduce computational linear 
algebra over the Laurent polynomial ring ${\Q}[v,v^{-1}]$ to linear algebra 
over the integers. Thus, generally speaking, we are faced with the
problem of devising efficient tools to do computational linear algebra 
over integral domains, not just over fields. In order to do so, we build
on general ideas from computational representation theory, more precisely
on the celebrated so-called \MA{} philosophy \cite{parker1}, which 
comprises of specially tailored, highly efficient techniques for
computational linear algebra over (small) finite fields.
Attempts to generalize these ideas to linear algebra
over the (infinite) field of rational numbers, and further to linear algebra 
over the integers have been coined the \IMA{} \cite{parker2}. The last word
on this has not been said yet, and in this paper we are trying to 
contribute here as well. (As future work, we are planning to develop
a full \IMA{} package along the present lines.)
But we are additionally going one step further by setting out to
extend these ideas to linear algebra over the univariate polynomial rings
over the rationals or the integers.

To do so, the basic idea is to reduce to linear algebra over the integers
by evaluating polynomials with rational coefficients at integral places, 
where we are using as few ``small''
places as possible, and to recover the polynomials in question by a 
Chinese remainder technique. Hence this strategy, fitting nicely into the
\IMA{} philosophy, differs from those known to the literature,
inasmuch we are neither using modular methods
(which would mean to go over to polynomial rings over finite fields),
nor are we in a position to use interpolation (which would mean
to use lots of places to evaluate at).
Thus another purpose of this paper is to give a detailed description of
the new computational tasks arising in pursuing this
strategy, and how we have accomplished them. Although the choice of the
material presented is governed by our application to Iwahori--Hecke algebras,
it is exhibited with a view towards general applicability.

\medskip
Here is an outline of the paper:
In {\bf Section \ref{sec:1}}
we recall some basic facts about representations 
of finite Coxeter groups and Iwahori--Hecke algebras, in particular
the notions of $W$-graphs, balancedness, and invariant bilinear forms. 
We conclude with Theorem \ref{balthm} saying that for 
the representations afforded by the $W$-graphs given by 
Naruse \cite{Naruse0}, Howlett and Yin \cite{How}, \cite{HowYin} 
are actually balanced, and in Tables \ref{Mmaxd0} and \ref{Mmaxd} 
we list some numerical data associated with their primitive Gram matrices. 

In the subsequent sections we describe our general approach towards 
linear algebra over integral domains, which consists of a cascade of steps:
In {\bf Section \ref{intcomp}} 
we first deal with linear algebra over $\Z$.
We discuss the key tasks of rational number recovery 
and of finding integral linear dependencies. 
Both tasks are known to the literature, but for the former 
we provide a variant containing a new feature, while for the 
latter we proceed along another strategy, within the \IMA{} philosophy.
Subsequently, we apply this to computing nullspaces, inverses, and the 
so-called ``exponents'' of matrices over $\Z$.
%
In {\bf Section \ref{polcomp}}
we then describe our general approach to deal with polynomials, in view
of our aim to do linear algebra over polynomial rings. The key task
is to recover a polynomial with rational coefficients from some of its 
evaluations at integral places. Here, we are aiming at using as few
``small'' places as possible, whence we are not in a position to apply
interpolation, but we are using a Chinese remainder technique instead.
Moreover, we devise a method to recover a polynomial from some of its 
evaluations where the latter are ``rescaled'' by unknown scalars; the 
necessity of being able to solve this task is closely related to
our use of the \IMA{}, hence to our knowledge this method is new as well. 
%
In {\bf Section \ref{matcomp}} 
we proceed to show how linear algebra over $\Z$ and polynomial recovery,
as discussed in earlier sections, can now be combined to do 
linear algebra over $\Z[X]$ and $\Q[X]$,
by devising methods to computing nullspaces, inverses, exponents and 
products of matrices using this new approach.
%
In {\bf Section \ref{repcomp}}
we finally recall the ``standard basis algorithm'' originally developed
in \cite{parker1} for computations over finite fields. We present a 
general variant for absolutely irreducible matrix representations over 
an arbitrary field, show how this can be used to compute homomorphisms
between such representations, and discuss how the necessary computations 
are facilitated over the fields $\Q$ and $\Q(X)$, using the tools we 
have developed. 

Having the general tools in place,
in {\bf Section \ref{bilcomp}}
we return to our particular application of computing Gram matrices 
of invariant bilinear forms for $W$-graph representations $\fX$ of 
Iwahori--Hecke algebras. We proceed along the strategy which has already 
been indicated in \cite[Section 4.3]{gemu}, where here we take the 
opportunity to provide full details. We begin by computing standard
bases for the representations $\fX$ and $\fX'$, where the latter is 
given by $\fX'(T_w):=\fX(T_{w^{-1}})^{\operatorname{tr}}$, for $w \in W$. 
In order to find suitable seed vectors to start with, we use an 
observation on restrictions of representations of Iwahori--Hecke algebras
to parabolic subalgebras, which naturally leads to certain distinguished
elements of $\cH$ having actions of co-rank one on $\fX$ and $\fX'$.
To actually run the standard basis algorithm subsequently, we again revert
to a specialization technique.
%
In {\bf Section \ref{bilcompII}}
we proceed by collecting a few observations on the standard bases
$B$ and $B'$ of the representations $\fX$ and $\fX'$ thus obtained.
Indeed, the matrix entries occurring seem to be much less arbitrary 
than expected from general principles, but this has only been verified 
experimentally for the representations under consideration here, while 
a priori proofs are largely missing (so far).
The final computational step then essentially is to determine the product
$B^{-1}\cdot B'$, which up to rescaling is a Gram matrix as desired.
To do this efficiently, apart from the general tools
developed above, we make heavy use of the special form 
of the matrix entries of $B^{-1}\cdot B'$ just mentioned.
%
In the concluding {\bf Section \ref{sec:timings}}
we provide running times and workspace requirements for our
computations in types $E_7$ and $E_8$, and present an explicit (tiny)
example for type $E_6$.

\medskip
It should be clear from the above description that 
to pursue our novel approach we had to solve quite a few tasks 
for which there was no pre-existing implementation, 
let alone in one and the same computer algebra system. 
To develop the necessary new code, as our computational platform 
we have chosen the computer algebra system \GAP{} \cite{GAP}.  
This system provides efficient arithmetics for the various basic objects 
we need: (i) rational integers and rational numbers, which in turn are 
handled by the {\sf GMP} library \cite{GMP}; (ii) row vectors and matrices 
over the integers, the rationals or (small) finite fields, where in this
context the entries of row vectors are actually treated as immediate objects;
(iii) floating point numbers, where the limited built-in facilities
are sufficient for our purposes.
Moreover, the necessary input data on Iwahori--Hecke algebras
and their representations is provided by the computer algebra system 
{\sf CHEVIE} \cite{jmich}, which conveniently is a branch of \GAP.

\medskip
\begin{center}
\begin{tabular}{lp{250pt}r} \multicolumn{3}{c}{\bf Contents} \\ \\
1 & Introduction \dotfill & \pageref{sec0}\\
2 & Iwahori--Hecke algebras and balanced representations \dotfill 
  & \pageref{sec:1}\\
3 & Linear algebra over the integers \dotfill & \pageref{intcomp}\\
4 & Computing with polynomials \dotfill & \pageref{polcomp}\\
5 & Linear algebra over polynomial rings \dotfill & \pageref{matcomp}\\
6 & Computing with representations \dotfill & \pageref{repcomp}\\
7 & Finding standard bases for $W$-graph representations \dotfill 
  & \pageref{bilcomp}\\
8 & Finding Gram matrices for $W$-graph representations \dotfill 
  & \pageref{bilcompII}\\
9 & Timings \dotfill & \pageref{sec:timings}\\
  & References \dotfill & \pageref{biblio}\\
\end{tabular}
\end{center}

\section{Iwahori--Hecke algebras and balanced representations} \label{sec:1}
We begin by recalling some basic facts about representations of finite
Coxeter groups and Iwahori--Hecke algebras; see \cite{gepf}, \cite{geja},
\cite{Lusztig03} for further details. 

\begin{abs} \label{Mabs11} We fix a finite Coxeter group $W$
with set of simple reflections~$S$; for $w\in W$, we denote by $l(w)$ the 
length of $w$ with respect to~$S$. Let $L\colon W \rightarrow \Z$ be a 
weight function as in \cite{Lusztig03}, that is, we have $L(ww')=L(w)+
L(w')$ whenever $w,w'\in W$ satisfy $l(ww')=l(w)+l(w')$. Such a 
weight function is uniquely determined by its values $L(s)$ for $s\in S$. 
We will assume throughout that 
\[ L(s)>0 \qquad\mbox{for all $s\in S$}.\]
Let $R\subseteq \C$ be a subring and $A=R[v,v^{-1}]$ be the ring of 
Laurent polynomials over $R$ in the indeterminate~$v$. Let $\cH=\cH_A(W,L)$ 
be the corresponding generic Iwahori--Hecke algebra. Thus, $\cH$ is an 
associative $A$-algebra which is free over $A$ with a basis $\{T_w\mid 
w\in W\}$; the multiplication is given by the following rule, where 
$s\in S$ and $w\in W$:
\[ T_sT_w=\left\{\begin{array}{cl} T_{sw} & \qquad\mbox{if $l(sw)=l(w)+1$},
\\ T_{sw}+(v^{L(s)}-v^{-L(s)})T_s & \qquad\mbox{if $l(sw)=l(w)-1$}.
\end{array}\right.\]
\end{abs}

\begin{abs} \label{Mabs12} 
Let $F\subseteq \C$ be the field of fractions of $R$ and assume that $F$ 
is a splitting field for $W$. (For example, we could take $R=F=\R$ since
$\R$ is known to be a splitting field for $W$.) Let $\Irr(W)$ be the set 
of simple $F[W]$-modules (up to isomorphism); we shall use the following 
notation:
\[ \Irr(W)=\{E^\lambda \mid \lambda\in\Lambda\}\qquad \mbox{and}\qquad 
d_\lambda=\dim E^\lambda\;\;(\lambda\in\Lambda),\]
where $\Lambda$ is a finite index set. Let $K=F(v)$ be the field of 
fractions of $A$ and $\cH_K=K\otimes_A \cH$ be the $K$-algebra obtained 
by extension of scalars from $A$ to $K$. Then $\cH_K$ is a split 
semisimple algebra and there is a bijection between $\Irr(W)$ and
$\Irr(\cH_K)$, the set of simple $\cH_K$-modules (up to isomorphism).
Given $\lambda \in\Lambda$, we denote by $E_v^\lambda$ a simple 
$\cH_K$-module corresponding to $E^\lambda$. Then $E_v^\lambda$ is
uniquely determined (up to isomorphism) by the following property. For 
$w\in W$, we have
\[ \mbox{trace}(T_w,E_v^\lambda)\in F[v,v^{-1}] \qquad\mbox{and}\qquad
\mbox{trace}(w,E^\lambda)=\mbox{trace}(T_w,E_v^\lambda)|_{v\mapsto 1}.\]
\end{abs}

\begin{abs} \label{Mabs13} The algebra $\cH_K$ is symmetric, with trace 
form $\tau\colon \cH_K \rightarrow K$ given by $\tau(T_1)=1$ and $\tau(T_w)
=0$ for $1\neq w\in W$. The basis dual to $\{T_w\mid w\in W\}$ is given by 
$\{T_{w^{-1}}\mid w\in W\}$. By the general theory of symmetric algebras, 
there are well-defined elements $0\neq \bc_\lambda\in A$ ($\lambda\in
\Lambda)$ such that the following orthogonality relations hold for $\lambda,
\mu\in\Lambda$:
\[\sum_{w\in W} \mbox{trace}(T_w,E_v^\lambda)\mbox{trace}(T_{w^{-1}},
E_v^\mu)=\left\{\begin{array}{cl} d_\lambda \bc_\lambda &\qquad 
\mbox{if $\lambda=\mu$},\\ 0 & \qquad\mbox{if $\lambda\neq \mu$}.
\end{array}\right.\]
As observed by Lusztig, we can write each $\bc_\lambda$ uniquely in the form 
\[ \bc_\lambda=f_\lambda v^{-2\ba_\lambda} +\mbox{linear combination of 
larger powers of $v$},\]
where $f_\lambda$ is a strictly positive real number and $\ba_\lambda$
is a non-negative integer. The ``$a$-invariants'' $\ba_\lambda$ will
play a major role in the sequel; these numbers are explicitly known
for all types of $W$ and all choices of $L$ (see \cite[\S 1.3]{geja},
\cite[Chap.~22]{Lusztig03}). Alternatively, $\ba_\lambda$ can be 
characterized as follows:
\[ \ba_\lambda=\min\{i\geq 0 \mid v^i\mbox{trace}(T_w,E_v^\lambda)
\in F[v]\mbox{ for all $w\in W$}\}.\]
\end{abs}

\begin{abs} \label{Mabs14}
Let $\cO\subseteq K$ be the localization of $F[v]$ in the prime ideal 
$(v)$, that is, $\cO$ consists of all fractions of the form $f/g\in K$ 
where $f,g\in F[v]$ and $g(0)\neq 0$. Let $\fX^\lambda \colon \cH_K
\rightarrow K^{d_\lambda\times d_\lambda}$ be a matrix representation
afforded by $E_v^\lambda$. Following \cite{myedin}, we say that
$\fX^\lambda$ is balanced if
\[ v^{\ba_\lambda} \fX^\lambda(T_w)\in 
\cO^{d_\lambda\times d_\lambda} \qquad \mbox{for all $w\in W$}.\] 
This concept plays a crucial role in the study of ``cellular structures''
on $\cH$ (see \cite{myedin}) and the determination of 
Kazhdan--Lusztig cells (see \cite[\S 4]{geha}). It is known that every 
$E_v^\lambda$ affords a balanced representation. Note that, given some 
matrix representation afforded by $E_v^\lambda$, the above condition is 
hard to verify since it involves representing matrices for {\it all} 
$w\in W$. Much better for practical purposes is the following condition.
\end{abs}

\begin{prop}[See \protect{\cite[Prop.~4.3, Remark~4.4]{myedin}}] 
\label{ediprop} Assume that $F\subseteq \R$. Let $\lambda\in\Lambda$ 
and $\fX^\lambda \colon \cH_K \rightarrow K^{d_\lambda\times d_\lambda}$ 
be a matrix 
representation afforded by $E_v^\lambda$. Then $\fX^\lambda$ is balanced 
if and only if there exists a symmetric matrix 
$\Omega^\lambda\in\GL_{d_\lambda}(\cO)$ such that 
\begin{equation*}
\Omega^\lambda\,\fX^\lambda(T_s)=\fX^\lambda(T_s)^{\operatorname{tr}}
\,\Omega^\lambda\qquad\mbox{for all $s \in S$}.\tag{$*$}
\end{equation*}
\end{prop}

\begin{rem} \label{edirem1} Note that, if a matrix $\Omega^\lambda$ 
satisfies ($*$), then it immediately follows that 
\[\Omega^\lambda\,\fX^\lambda(T_{w^{-1}})=\fX^\lambda
(T_w)^{\operatorname{tr}}\,\Omega^\lambda\qquad\mbox{for all $w\in W$}.\]
Thus, $\Omega^\lambda$ is the Gram matrix of a symmetric 
bilinear form $\langle\;,\;\rangle_\lambda\colon E_v^\lambda\times 
E_v^\lambda \rightarrow K$ which is $\cH_K$-invariant in the sense that 
\[ \langle T_w.e,e'\rangle_\lambda=\langle e,T_{w^{-1}}.e'\rangle_\lambda
\qquad\mbox{for all $e,e'\in E_v^\lambda$ and $w\in W$}.\]
\end{rem}

\begin{rem} \label{edirem2} Assume that $F\subseteq\R$. Let $\lambda\in
\Lambda$ and $\fX^\lambda \colon \cH_K \rightarrow 
K^{d_\lambda\times d_\lambda}$ be a 
matrix representation afforded by $E_v^\lambda$. Let $\cE(\fX^\lambda)$ be 
the set of all $P\in K^{d_\lambda\times d_\lambda}$ such that 
$P\,\fX^\lambda(T_s)=
\fX^\lambda(T_s)^{\operatorname{tr}} \,P$ for $s \in S$. Since $\fX^\lambda$
is irreducible, Schur's Lemma implies that all matrices in $\cE(\fX^\lambda)$
are scalar multiples of each other. By \cite[Remark~1.4.9]{geja}, there
is a specific element $P_0\in \cE(\fX^\lambda)$ given by 
\[ P_0:=\sum_{w \in W} \fX^\lambda(T_w)^{\text{tr}}\, \fX^\lambda(T_w)
\in K^{d_\lambda\times d_\lambda};\]
furthermore, we have $\det(P_0)\neq 0$. By the Schur Relations (see 
\cite[7.2.1]{gepf}), we have
\[\sum_{w \in W} \fX^\lambda(T_{w^{-1}})P_0^{-1} \fX^\lambda(T_w)=
\mbox{trace}(P_0^{-1})\bc_\lambda I_{d_\lambda}.\]
Using the relation $P_0\,\fX^\lambda(T_{w^{-1}})=\fX^\lambda
(T_w)^{\operatorname{tr}} \,P_0$ for all $w\in W$, we deduce that
\[\mbox{trace}(P_0^{-1})\bc_\lambda=1.\]
This provides a direct criterion for checking if a given matrix
$P\in \cE(\fX^\lambda)$ equals~$P_0$. Furthermore, if $P\neq 0$ is an
element of $\cE(\fX^\lambda)$, then $P=cP_0$ for some $0\neq c\in K$ and 
so $\bc_\lambda\mbox{trace}(P^{-1})P=\bc_\lambda\mbox{trace}(P_0^{-1})P_0=
P_0$.
\end{rem}

The following concept was introduced by Kazhdan--Lusztig \cite{KaLu} in
the equal parameter case (where $L(s)=1$ for all $s\in S$); for the general 
case see \cite[\S 1.4]{geja}.

\begin{defn} \label{wgraph} Let $V$ be an $\cH_K$-module with $d:=\dim V
<\infty$. We say that $V$ is {\it afforded by a $W$-graph} if there
exist 
\begin{itemize}
\item a basis $\{e_1,\ldots,e_d\}$ of $V$, 
\item subsets $I_i\subseteq S$ for $1\leq i\leq d$,
\item and elements $m_{ij}^s\in A$, where $1\leq i,j
\leq d$ and $s \in I_i\setminus I_j$, 
\end{itemize}
such that the following hold. First, we require that
\[v^{L(s)}m_{ij}^s\in vR[v]\quad\mbox{and}\quad m_{ij}^s=
m_{ij}^s|_{v\mapsto v^{-1}} \quad \mbox{for all $1\leq i,j\leq d$, 
$s \in I_i\setminus I_j$}.\]
Furthermore, for $s\in S$, the action of $T_s$ on $V$ is given by
\begin{equation*}
\renewcommand{\arraystretch}{1.2}
T_s.e_j= \left\{\begin{array}{ll} \displaystyle{v^{L(s)}\,e_j+ 
\sum_{1\leq i \leq d:\,s \in I_i} m_{ij}^s \,e_i}
&\qquad\mbox{if $s\not\in I_j$},\\ -v^{-L(s)}\,e_j &\qquad 
\mbox{if $s \in I_j$}.\end{array}\right.
\end{equation*}
Thus, if $V$ is afforded by a $W$-graph representation, then the action of 
$T_s$ on $V$ is given by matrices of a particularly simple form. 
\end{defn}

It has been conjectured in \cite{myedin} (see also \cite[1.4.14]{geja})
that, if the simple $\cH_K$-module $E_v^\lambda$ is afforded by a 
$W$-graph, then the corresponding matrix representation is balanced. We
now turn to the problem of explicitly verifying if a given irreducible 
matrix representation of $\cH_K$ is balanced or not. 

\begin{abs} \label{wgr1a} We shall assume from now that $W$ is a finite 
Weyl group and that we are in the equal parameter case where $L(s)=1$ 
for all $s\in S$; we may take $R=\Z$, $F=\Q$ in the above discussion. (The 
remaining cases have been dealt with in \cite[Examples~4.5, 4.6]{myedin}.)
It is known that every simple $\cH_K$-module $E_v^\lambda$ is afforded by 
a $W$-graph; see \cite[Theorem~2.7.2]{geja} and the references there. As 
far as $W$ of exceptional type is concerned, such $W$-graphs have been 
determined explicitly, by Naruse \cite{Naruse0}, Howlett and Yin \cite{How}, 
\cite{HowYin}. They are available in electronic form through Michel's 
development version of the {\sf CHEVIE} system; see \cite{jmich}. Now let us 
fix $\lambda\in\Lambda$ and assume that $\fX^\lambda\colon \cH_K \rightarrow 
K^{d_\lambda\times d_\lambda}$ is
a corresponding representation afforded by a $W$-graph. Concretely, this 
will mean that we are given the collection of matrices $\{X_s:=
\fX^\lambda(T_s)\mid s\in S\}$. Our aim is to find a matrix $P=
(p_{ij})_{1\leq i,j\leq d_\lambda}$ such that 
\begin{equation*}
PX_s=X_s^{\operatorname{tr}} \,P\qquad\mbox{for all $s \in S$}.\tag{$*$}
\end{equation*}
This is a system of $|S|d_\lambda^2$ homogeneous linear equations for the 
$d_\lambda(d_\lambda+1)/2$ unknown entries of $P$. (Recall that $P$ is
symmetric.) We know that $P$ is uniquely determined up to scalar multiples. 
Rescaling a given solution by a suitable non-zero polynomial in $\Q[v]$, 
we can assume that all entries of $P$ are in $\Z[v]$ and that their 
greatest common divisor is $\pm 1$; then $P$ is unique up to a sign.
Such a solution $P$ will be called a {\em primitive Gram matrix} for 
$\fX^\lambda$. As in \ref{edirem2}, a specific solution $P_0$ can be 
singled out by the condition that $\mbox{trace}(P_0^{-1})\bc_\lambda=1$. 
We claim that
\begin{itemize}
\item the matrix $P_0':=v^{2l(w_0)}P_0$ has entries in $\Z[v]$, and
\item the non-zero entries of $P_0'$ have degree at most $2l(w_0)$.
\end{itemize}
Here, $w_0$ denotes the longest element of $W$. Indeed, since all the 
entries of the matrices $X_s$ ($s\in S$) are in $\Z[v,v^{-1}]$, the same 
will be true for $P_0$ as well. The formulae in \ref{wgraph} show that each
matrix $vX_s$ ($s \in S$) has entries in $\Z[v]$. Hence, all matrices 
$v^{l(w_0)} \fX^\lambda(T_w)$ have entries in $\Z[v]$ and so $P_0'$ has 
entries in $\Z[v]$. Furthermore, the non-zero entries of each matrix $vX_s$
have degree $0$, $1$ or $2$. This yields the degree bound for the entries
of~$P_0'$.

Since the entries of $P_0'$ are integer polynomials of bounded degree,
we can determine $P_0'$ by interpolation and modular techniques (Chinese
remainder). Combining this with the techniques described in 
\cite[\S 4.3]{gemu}, one obtains an algorithm which can be implemented 
in {\sf GAP} in a straightforward way. 
Rescaling these matrices by suitable non-zero polynomials in $\Q[v]$,
we obtain primitive Gram matrices as solutions of~($*$). 
This approach readily produces primitive Gram matrices for $W$ 
of type $F_4$, $E_6$ and $E_7$ in a few hours of computing time.
As was already advertised in Section~\ref{sec0}, we also succeeded
in obtaining primitive Gram matrices for type $E_8$, where it is 
one of the purposes of this paper to describe the methods involved.
\end{abs}

Tables~\ref{Mmaxd0} and~\ref{Mmaxd} contain some information 
about these primitive Gram matrices $P$: 
\begin{center}\begin{tabular}{ll}
1st column: &usual names of the irreducible representations.\\
2nd column: &maximum degree of a non-zero entry of $P$.\\
3rd column: &maximum absolute value of a coefficient of an
entry of $P$.\\
4th column: &is the specialized matrix $P|_{v\rightarrow 0}$ diagonal?\\
5th column: &prime divisors of the determinant of $P|_{v\rightarrow 0}$.\\
& (No entry means that this determinant is $\pm 1$.) \end{tabular}
\end{center}
%
%
We note that the primes in the 5th column are so-called ``bad primes'' for 
$W$ (as in \cite[1.5.11]{geja}).
In particular, the fact that $P|_{v\rightarrow 0}$ always has a non-zero 
determinant means that $\det(P)\in\cO^\times$ (see Proposition~\ref{ediprop}).
Thus, we can conclude:

\begin{thm} \label{balthm} Let $W$ be of type $F_4$, $E_6$, $E_7$
or $E_8$ and $L(s)=1$ for all $s\in S$. Then the $W$-graph
representations of Naruse \cite{Naruse0}, Howlett and Yin  \cite{How}, 
\cite{HowYin} are balanced.
\end{thm}

\afterpage{\begin{landscape}
\begin{table}[ph!]
\caption{Information on primitive Gram matrices for type $F_4$, $E_6$, 
$E_7$; cf.\ \ref{wgr1a}} \label{Mmaxd0}
\begin{center} 
$\begin{array}{c@{\hspace{4pt}}c@{\hspace{4pt}}c@{\hspace{4pt}}
c@{\hspace{4pt}}c} \hline \text{$F_4$} & \text{deg.} & \text{abs.\
val.} & \text{diag.} & \text{det}\\ \hline
1_1 & 0 & 1 & y & \\
1_2 & 0 & 1 & y & \\
1_3 & 0 & 1 & y & \\
1_4 & 0 & 1 & y & \\
2_1 & 2 & 1 & y & \\
2_2 & 2 & 1 & y & \\
2_3 & 2 & 1 & y & \\
2_4 & 2 & 1 & y & \\
4 & 4 & 2 & y & \\
9_1 & 4 & 2 & y & \\
9_2 & 6 & 8 & y & 2\\
9_3 & 6 & 12 & y & 2\\
9_4 & 10 & 6 & y & \\
6_1 & 4 & 2 & y & \\
6_2 & 4 & 4 & y & 2\\
12 & 8 & 54 & y & 2,3\\
4_1 & 2 & 2 & y & 2\\
4_2 & 2 & 1 & y & \\
4_3 & 2 & 1 & y & \\
4_4 & 6 & 4 & y & 2\\
8_1 & 4 & 2 & y & \\
8_2 & 6 & 3 & y & \\
8_3 & 4 & 2 & y & \\
8_4 & 6 & 3 & y & \\
16 & 8 & 16 & y & 2\\
\hline \\\\\\\\\end{array}\qquad\qquad
\begin{array}{c@{\hspace{4pt}}c@{\hspace{4pt}}c@{\hspace{4pt}}
c@{\hspace{4pt}}c} \hline \text{$E_6$} & \text{deg.} & \text{abs.\
val.} & \text{diag.} & \text{det}\\ \hline
1_p & 0 & 1 & y & \\
1_p' & 0 & 1 & y & \\
10_s & 6 & 3 & y & \\
6_p & 2 & 1 & y & \\
6_p' & 10 & 5 & y & \\
20_s & 6 & 3 & y & \\
15_p & 4 & 2 & y & \\
15_p' & 8 & 4 & y & \\
15_q & 6 & 3 & y & \\
15_q' & 8 & 8 & y & \\
20_p & 4 & 2 & y & \\
20_p' & 22 & 61 & y & \\
24_p & 6 & 5 & y & \\
24_p' & 12 & 20 & y & \\
30_p & 6 & 6 & y & 2\\
30_p' & 18 & 304 & y & 2\\
60_s & 10 & 26 & y & \\
80_s & 14 & 711 & y & 2,3\\
90_s & 12 & 58 & n & 3\\
60_p & 10 & 21 & y & \\
60_p' & 12 & 44 & y & \\
64_p & 8 & 12 & y & \\
64_p' & 20 & 192 & y & \\
81_p & 10 & 24 & y & \\
81_p' & 12 & 32 & y & \\
\hline\\\\\\\\ \end{array}\qquad\qquad
\begin{array}{c@{\hspace{4pt}}c@{\hspace{4pt}}c@{\hspace{4pt}}
c@{\hspace{4pt}}c} \hline \text{$E_7$} & \text{deg.} & \text{abs.\
val.} & \text{diag.} & \text{det}\\ \hline
1_a & 0 & 1 & y & \\
1_a' & 0 & 1 & y & \\
7_a & 12 & 6 & y & \\
7_a' & 2 & 1 & y & \\
15_a & 8 & 8 & y & \\
15_a' & 6 & 3 & y & \\
21_a & 4 & 2 & y & \\
21_a' & 10 & 7 & y & \\
21_b & 16 & 19 & y & \\
21_b' & 4 & 2 & y & \\
27_a & 4 & 2 & y & \\
27_a' & 32 & 164 & y & \\
35_a & 8 & 6 & y & \\
35_a' & 6 & 3 & y & \\
35_b & 6 & 3 & y & \\
35_b' & 22 & 144 & y & \\
56_a & 26 & 1082 & y & 2\\
56_a' & 6 & 6 & y & 2\\
70_a & 12 & 56 & y & \\
70_a' & 10 & 26 & y & \\
84_a & 10 & 26 & y & \\
84_a' & 16 & 148 & y & \\
105_a & 22 & 377 & y & \\
105_a' & 8 & 12 & y & \\
105_b & 10 & 21 & y & \\
105_b' & 20 & 504 & y & \\
105_c & 12 & 38 & y & \\
105_c' & 12 & 44 & y & \\
120_a & 8 & 24 & y & 2\\
120_a' & 30 & 7516 & y & 2\\
\hline \end{array}\qquad\quad
\begin{array}{c@{\hspace{3pt}}c@{\hspace{3pt}}c@{\hspace{3pt}}
c@{\hspace{3pt}}c} \hline \text{$E_7$} & \text{deg.} & \text{abs.\
val.} & \text{diag.} & \text{det}\\ \hline
168_a & 10 & 35 & y & \\
168_a' & 26 & 2193 & y & \\
189_a & 12 & 56 & y & \\
189_a' & 16 & 112 & y & \\
189_b & 28 & 7498 & y & \\
189_b' & 10 & 42 & y & \\
189_c & 22 & 454 & y & \\
189_c' & 10 & 38 & y & \\
210_a & 10 & 35 & y & \\
210_a' & 22 & 973 & y & \\
210_b & 14 & 253 & y & \\
210_b' & 16 & 468 & y & \\
216_a & 22 & 1596 & y & \\
216_a' & 14 & 227 & y & \\
280_a & 22 & 1836 & n & 3\\
280_a' & 12 & 58 & n & 3\\
280_b & 14 & 241 & y & \\
280_b' & 20 & 2368 & y & \\
315_a & 26 & 47277 & y & 2,3\\
315_a' & 14 & 4122 & y & 2,3\\
336_a & 20 & 892 & y & \\
336_a' & 14 & 175 & y & \\
378_a & 24 & 7310 & y & \\
378_a' & 14 & 453 & y & \\
405_a & 14 & 637 & y & 2\\
405_a' & 26 & 46878 & y & 2\\
420_a & 16 & 1332 & y & 2\\
420_a' & 20 & 4148 & y & 2\\
512_a & 20 & 6036 & y & \\
512_a' & 20 & 6036 & y & \\
\hline\end{array}$
\end{center}
\end{table}
\end{landscape}}

\afterpage{\begin{landscape}
\begin{table}[ph!]
\caption{Information on primitive Gram matrices for type $E_8$;
cf.\ \ref{wgr1a}} \label{Mmaxd}
\begin{center} 
$\begin{array}{c@{\hspace{1pt}}c@{\hspace{3pt}}c@{\hspace{3pt}}
c@{\hspace{3pt}}c} \hline \text{repr.} & \text{deg.} & \text{abs.\
val.} & \text{diag.} & \text{det}\\ \hline
1_x & 0 & 1 & y & \\
1_x' & 0 & 1 & y & \\
28_x & 4 & 2 & y & \\
28_x' & 12 & 10 & y & \\
35_x & 4 & 2 & y & \\
35_x' & 38 & 377 & y & \\
70_y & 8 & 6 & y & \\
50_x & 8 & 6 & y & \\
50_x' & 22 & 257 & y & \\
84_x & 6 & 3 & y & \\
84_x' & 38 & 675 & y & \\
168_y & 16 & 340 & y & \\
175_x & 12 & 52 & y & \\
175_x' & 20 & 992 & y & \\
210_x & 8 & 24 & y & 2\\
210_x' & 42 & 95780 & y & 2\\
420_y & 16 & 1432 & y & \\
300_x & 10 & 41 & y & \\
300_x' & 40 & 12710 & y & \\
350_x & 12 & 56 & y & \\
350_x' & 20 & 290 & y & \\
525_x & 12 & 76 & y & \\
525_x' & 24 & 1946 & y & \\
567_x & 10 & 54 & y & \\
567_x' & 42 & 57812 & y & \\
1134_y & 22 & 8739 & y & \\
700_xx & 18 & 1399 & y & \\
700_xx' & 20 & 5982 & y & \\
\hline \end{array}\qquad\quad
\begin{array}{c@{\hspace{1pt}}c@{\hspace{1pt}}c@{\hspace{1pt}}
c@{\hspace{1pt}}c} \hline \text{repr.} & \text{deg.} & \text{abs.\
val.} & \text{diag.} & \text{det}\\ \hline
700_x & 12 & 538 & y & 2\\
700_x' & 54 & 16489188 & y & 2\\
1400_y & 22 & 22286 & n & 2,3\\
840_x & 16 & 6044 & y & \\
840_x' & 26 & 37603 & y & \\
1680_y & 22 & 3447 & n & 2,5\\
972_x & 16 & 2098 & y & \\
972_x' & 36 & 185342 & y & \\
1050_x & 16 & 3792 & y & \\
1050_x' & 34 & 390765 & y & \\
2100_y & 22 & 5561 & y & \\
1344_x & 14 & 1140 & y & \\
1344_x' & 40 & 381082 & y & \\
2688_y & 24 & 169180 & y & \\
1400_x & 16 & 41820 & y & 2,3\\
1400_x' & 48 & 763453596 & y & 2,3\\
1575_x & 14 & 783 & n & 3\\
1575_x' & 44 & 850956 & n & 3\\
3150_y & 26 & 6166994 & y & 2\\
2100_x & 20 & 3514 & y & \\
2100_x' & 26 & 12511 & y & \\
4200_y & 28 & 58249760 & n & 2\\
2240_x & 20 & 1878156 & y & 2\\
2240_x' & 42 & 60390945 & y & 2\\
4480_y & 32 & 85556320920 & y & 2,3,5\\
2268_x & 16 & 5948 & y & 2\\
2268_x' & 40 & 6442224 & y & 2\\
4536_y & 28 & 3887856 & n & 2\\
\hline \end{array}\qquad\quad
\begin{array}{c@{\hspace{1pt}}c@{\hspace{1pt}}c@{\hspace{1pt}}
c@{\hspace{1pt}}c} \hline \text{repr.} & \text{deg.} & \text{abs.\
val.} & \text{diag.} & \text{det}\\ \hline
2835_x & 24 & 1344484 & y & \\
2835_x' & 32 & 5391418 & y & \\
5670_y & 30 & 10762741 & n & 2,3,5\\
3200_x & 24 & 266284 & y & \\
3200_x' & 30 & 587345 & y & \\
4096_x & 22 & 531634 & y & \\
4096_x' & 44 & 234956568 & y & \\
4200_x & 24 & 5413484 & y & 2\\
4200_x' & 36 & 129331224 & y & 2\\
6075_x & 26 & 894864 & y & \\
6075_x' & 34 & 10488013 & y & \\
8_z & 2 & 1 & y & \\
8_z' & 14 & 6 & y & \\
56_z & 6 & 3 & y & \\
56_z' & 10 & 7 & y & \\
112_z & 6 & 6 & y & 2\\
112_z' & 54 & 20790 & y & 2\\
160_z & 8 & 12 & y & \\
160_z' & 32 & 400 & y & \\
448_w & 16 & 128 & y & \\
400_z & 12 & 132 & y & \\
400_z' & 38 & 58368 & y & \\
448_z & 12 & 290 & y & \\
448_z' & 32 & 17290 & y & \\
560_z & 10 & 73 & y & \\
560_z' & 46 & 408409 & y & \\
1344_w & 24 & 177956 & y & \\
840_z & 14 & 643 & y & \\
\hline \end{array}\qquad\quad
\begin{array}{c@{\hspace{1pt}}c@{\hspace{1pt}}c@{\hspace{1pt}}
c@{\hspace{1pt}}c} \hline \text{repr.} & \text{deg.} & \text{abs.\
val.} & \text{diag.} & \text{det}\\ \hline
840_z' & 26 & 8048 & y & \\
1008_z & 12 & 156 & n & 3\\
1008_z' & 40 & 66780 & n & 3\\
2016_w & 28 & 797422 & y & \\
1296_z & 14 & 345 & y & \\
1296_z' & 34 & 23195 & y & \\
1400_{zz} & 16 & 10042 & y & \\
1400_{zz}' & 34 & 358379 & y & \\
1400_z & 14 & 8148 & y & 2,3\\
1400_z' & 50 & 60122676 & y & 2,3\\
2400_z & 22 & 6380 & y & \\
2400_z' & 28 & 55922 & y & \\
2800_z & 20 & 38038 & y & 2\\
2800_z' & 30 & 882222 & y & 2\\
5600_w & 26 & 372230 & n & 3\\
3240_z & 16 & 25586 & y & \\
3240_z' & 48 & 33653538 & y & \\
3360_z & 20 & 29722 & y & \\
3360_z' & 32 & 775084 & y & \\
7168_w & 32 & 1190470476 & y & 2,3\\
4096_z & 22 & 531634 & y & \\
4096_z' & 44 & 234956568 & y & \\
4200_z & 26 & 728053 & y & \\
4200_z' & 28 & 1298612 & y & \\
4536_z & 24 & 2728756 & y & \\
4536_z' & 38 & 50779421 & y & \\
5600_z & 26 & 3115126 & y & 2\\
5600_z' & 30 & 3848044 & y & 2\\
\hline\end{array}$
\end{center}
\end{table}
\end{landscape}}

\section{Linear algebra over the integers}\label{intcomp}

As was already mentioned in Section \ref{sec0}, the basic strategy
of our approach to determine Gram matrices of invariant bilinear forms 
for representations of Iwahori--Hecke algebras is to reduce computational 
linear algebra over the polynomial rings $\Z[X]$ or $\Q[X]$, where from 
now on $X$ denotes our favorite indeterminate, to computational linear 
algebra over the integers $\Z$. Thus in this section we begin by 
describing how we deal with matrices over $\Z$, where we restrict 
ourselves to the aspects needed for our present application. 

Let us fix the following convention: 
For $x,y\in\Z$, not both zero, let $\gcd(x,y)\in\Z$ denote the positive 
greatest common divisor of $x$ and $y$. A vector $0\neq v\in\Q^m$,
where $m\in\N$, is called \emph{primitive}, if actually $v\in\Z^m$, 
and for the greatest common divisor $\gcd(v)$ of its entries we have
$\gcd(v)=1$. Clearly greatest common divisor computations in $\Z$
yield a $\Q$-multiple of $v$ which is primitive.
\nopagebreak
Similarly, a matrix $0\neq A\in\Z^{m\times n}$, where $m,n\in\N$,
is called \emph{primitive}, if actually $A\in\Z^{m\times n}$, and for the
greatest common divisor $\gcd(A)$ of its entries we have $\gcd(A)=1$. 

\clearpage
\absT{Continued fractions and the Euclidean algorithm}
The first computational task we are going to discuss,
in Section \ref{ratrecog} below, is rational number recovery.
This has been discussed in the literature at various places,
see for example \cite{dixon,mon,parker2} or \cite[Section 5.10]{vzG}.
(We also gratefully acknowledge additional private discussions 
with R.~Parker on this topic.) 
Although the ideas pursued in these references are closely
related to ours, none of them completely coincides with our approach,
and proofs (if given at all) are not too elucidating.
Hence we present our approach in detail, for which we need
a few preparations first:
\end{abs}

\subT{Continued fraction expansions}
We recall a few notions from the theory of continued fraction expansions;
as a general reference see for example \cite[Chapter 10]{hw}:
Given $\rho\in\R$ such that $\rho\geq 0$, let
$$ \textrm{cf}[q_1,q_2,\ldots] 
  =q_1+\frac{1}{q_2+\frac{1}{\ddots}} $$
be its \emph{(regular) continued fraction} expansion,
where $q_1\in\N_0$ and $q_i\in\N$ for $i\geq 2$. This
is obtained by letting $q_1:=\lfloor{\rho}\rfloor$, and,
as long as $\rho\neq q_1$, proceeding recursively with 
$\frac{1}{\rho-q_1}$ instead of $\rho$. This process terminates,
after $l\geq 1$ steps say, if and only if $\rho\in\Q$; otherwise we 
let $l:=\infty$. 
Truncating at $i\leq l$ yields the $i$-th \emph{convergent} 
$\rho_i:=\textrm{cf}[q_1,\ldots,q_i]\in\Q$ of $\rho$, 
hence we may write 
$\rho_i:=\frac{\sigma_i}{\tau_i}$, where $\sigma_i,\tau_i\in\N_0$ 
such that $\tau_i\geq 1$ and $\gcd(\sigma_i,\tau_i)=1$.
Letting additionally $\sigma_{-1}:=0$ and $\tau_{-1}:=1$, as well as
$\sigma_0:=1$ and $\tau_0:=0$, for $i\geq 1$ we get by induction 
$$ \sigma_i=q_i\sigma_{i-1}+\sigma_{i-2}
\quad\textrm{and}\quad 
\tau_i=q_i\tau_{i-1}+\tau_{i-2} .$$
Hence the sequences
$[\sigma_1,\sigma_2,\ldots,\sigma_l]$ and $[\tau_2,\tau_3,\ldots,\tau_l]$ 
are strongly increasing.

Now let $\rho=\frac{a}{b}\in\Q$, where $a,b\in\N$.
Then the continued fraction expansion of $\rho$
can be computed by the extended Euclidean algorithm, see
\cite[Algorithm 1.3.6]{cohen}, as follows:
Setting $r_0:=a$ and $r_1:=b$, 
for $1\leq i\leq l$ let recursively $q_i\in\N_0$ and 
$$ r_{i+1}:=r_{i-1}-q_ir_i\in\N_0 
\quad\textrm{such that}\quad 
r_{i+1}<r_i, $$
where $l\geq 1$ is defined by $r_l>0$ but $r_{l+1}=0$;
actually we have $q_i\geq 1$ for $i\geq 2$, and of course $r_l=\gcd(a,b)$. 
Hence the sequence $[r_1,\ldots,r_{l+1}]$ has
non-negative entries and is strongly decreasing.
Moreover, setting $s_0:=1$ and $t_0:=0$, as well as $s_1:=0$ and $t_1:=1$, 
and for $1\leq i\leq l$ letting recursively 
$$ s_{i+1}:=s_{i-1}-q_is_i
\quad\textrm{and}\quad 
t_{i+1}:=t_{i-1}-q_it_i ,$$
we get $r_i=s_ia+t_ib$.
Then it is immediate by induction that
$\sigma_i=(-1)^i\cdot t_{i+1}$ and $\tau_i=(-1)^{i+1}\cdot s_{i+1}$,
for $i\geq 1$, and hence 
$$ \rho_i = - \frac{t_{i+1}}{s_{i+1}}, 
\quad\textrm{where}\quad 
\gcd(s_{i+1},t_{i+1})=1, 
\quad\textrm{for}\quad 
1\leq i\leq l. $$
Hence the sequences $[-s_3,s_4,-s_5\ldots,\pm s_{l+1}]$ and
$[-t_2,t_3,-t_4\ldots,\pm t_{l+1}]$ 
have positive entries and are strongly increasing. 
Finally, a direct computation yields
$$ \rho-\rho_i 
=\frac{a}{b}-\frac{\sigma_i}{\tau_i}
=\frac{\tau_i a-\sigma_i b}{\tau_i b}
=\frac{s_{i+1} a+t_{i+1}b}{s_{i+1} b}
=\frac{r_{i+1}}{bs_{i+1}},
\quad\textrm{for}\quad
1\leq i\leq l .$$
\end{sub}

\subT{Another view on the Euclidean algorithm}
For $a,b\in\N$ we consider the $\Z$-lattice
$$ L_{a,b}:=\langle[1,a],[0,b]\rangle_\Z\subseteq\Z^2 .$$ 
Then we have $|\det(L_{a,b})|=b$, and it is immediate that 
$[x,y]\in\Z^2$ is an element of $L_{a,b}$ if and only if $y\equiv ax\pmod{b}$.
Note that if $0\neq [x,y]\in L_{a,b}$ is primitive, 
then we necessarily have $\gcd(x,b)=1$.
Moreover, the extended Euclidean algorithm shows that
$L_{a,b}=\langle[s_i,r_i],[s_{i+1},r_{i+1}]\rangle_\Z$, for all 
$0\leq i\leq l$. We collect a few properties of $L_{a,b}$:
\end{sub}

\begin{lem}\label{lattice}
(a)
For all $0\leq i\leq l+1$ we have 
$\langle[s_i,r_i]\rangle_\Q\cap L_{a,b}=\langle[s_i,r_i]\rangle_\Z$.

(b)
We have $\langle[s_i,r_i]\rangle_\Q=\langle[s_j,r_j]\rangle_\Q$,
where $1\leq i,j\leq l+1$, if and only if $i=j$.
\end{lem}

\begin{proof} 
We first show that whenever $[x,y]\in L_{a,b}$ such that $0<|y|<r_i$, 
for some $0\leq i\leq l$, then $|x|\geq|s_{i+1}|$:
We may assume that $i\geq 2$. Let $c,d\in\Z$ such that 
$$ [x,y]=[c,d]\cdot
\begin{bmatrix} s_i & r_i \\ s_{i+1} & r_{i+1} \\ \end{bmatrix} ,$$
where we may assume that $c\neq 0$, which entails $d\neq 0$ as well. 
Since $r_i>r_{i+1}\geq 0$, this implies $c\cdot d<0$.
Since the sequence $[s_2,-s_3,s_4,-s_5\ldots,\pm s_{l+1}]$ has
positive entries, we get 
$|x|=|cs_i+ds_{i+1}|=|c|\cdot|s_i|+|d|\cdot|s_{i+1}|\geq |s_{i+1}|$,
as asserted.

(a)
We may assume that $i\geq 2$. Moreover, for $i=l+1$ letting
$[x,0]\in L_{a,b}$, it is immediate from $ax\equiv 0\pmod{b}$ that 
$|s_{l+1}|=\frac{b}{r_l}=\frac{b}{\gcd(a,b)}$ divides $x$.
Hence we may assume $i\leq l$, too. Then let 
$d\neq 1$ be a divisor of $\gcd(s_i,r_i)$ such that 
$\frac{1}{d}\cdot[s_i,r_i]\in L_{a,b}$. Then we have $0<|\frac{r_i}{d}|<r_i$
and $|\frac{s_i}{d}|<|s_i|\leq|s_{i+1}|$, contradicting the statement above.

(b)
It follows from (a) that there are $c,d\in\Z$ such that
$[s_j,r_j]=c\cdot [s_i,r_i]$ and $[s_i,r_i]=d\cdot [s_j,r_j]$.
Hence we get $cd=1$, and since the sequence $[r_1,\ldots,r_{l+1}]$ has
non-negative entries and is strongly decreasing, we infer $r_i=r_j$ and $i=j$.
\qed\end{proof}

Note that the statement in (b) is trivial if $[s_i,r_i]$ is primitive,
that is $\gcd(s_i,r_i)=1$. But this is not always fulfilled, as the
example in \cite[Example 5.27]{vzG} shows.

\begin{prop}\label{minimum} 
(a)
Let $[x,y]\in L_{a,b}$ such that $x\neq 0$ and $|x|\cdot |y|\leq\frac{b}{2}$. 
Then we have $[x,y]\in\langle[s_i,r_i]\rangle_\Z$, for a unique
$2\leq i\leq l+1$. In particular, if $[x,y]$ is primitive
then we have $[x,y]=[s_i,r_i]$ or $[x,y]=-[s_i,r_i]$.

(b)
Assume there is $0\neq [x,y]\in L_{a,b}$ such that 
$\|[x,y]\|:=\sqrt{x^2+y^2}<\sqrt{b}$.
Then there is a unique $2\leq i\leq l+1$ such that
$\|[s_i,r_i]\|<\sqrt{b}$, and the shortest non-zero elements of $L_{a,b}$ 
are precisely $[s_i,r_i]$ and $-[s_i,r_i]$.
\end{prop}

\begin{proof} 
(a)
Since $[x,y]\in L_{a,b}$ there is $z\in\Z$ such that $y=xa-zb$. Then we have 
$$ \Big|\frac{a}{b}-\frac{z}{x}\Big| = \frac{|y|}{b\cdot|x|}
= \frac{|x|\cdot|y|}{b\cdot|x|^2} \leq \frac{1}{2\cdot|x|^2} .$$
Thus by Legendre's Theorem, see \cite[Section 10.15, Theorem 184]{hw}, 
we infer that $\frac{z}{x}$ occurs as a convergent in the continued
fraction expansion of $\rho=\frac{a}{b}$, that is, there is
$2\leq i\leq l+1$ such that $\frac{z}{x}=\rho_{i-1}$. This yields
$$ \frac{y}{x} = \frac{xa-zb}{x} = a-\frac{zb}{x} 
=a-b\rho_{i-1}=b(\rho-\rho_{i-1})=\frac{r_i}{s_i} .$$
Hence we have $[x,y]\in\langle[s_i,r_i]\rangle_\Q$, and thus
from Lemma \ref{lattice} we get $[x,y]\in\langle[s_i,r_i]\rangle_\Z$,
together with the uniqueness statement.

(b)
Assume first that $x=0$, then by Lemma \ref{lattice} we infer
that $b$ divides $y$, and hence $\|[x,y]\|\geq b\geq\sqrt{b}$, a 
contradiction. Hence we have $x\neq 0$.
Moreover, from $(x-y)^2=x^2+y^2-2xy\geq 0$ we get 
$2\cdot|x|\cdot|y|\leq x^2+y^2=\|[x,y]\|^2<b$,
hence from (a) we see that there is $2\leq i\leq l+1$ such that 
$[x,y]=\langle[s_i,r_i]\rangle_\Z$. Thus in particular we have 
$\|[s_i,r_i]\|<\sqrt{b}$. 

In order to show uniqueness, and the statement on shortest elements, 
let $0\neq [x',y']\in L_{a,b}$ such that 
$\|[x',y']\|<\sqrt{b}$. Then, as above, there is $2\leq i\leq l+1$ 
such that $[x',y']=\langle[s_j,r_j]\rangle_\Z$, hence in particular we have
$\|[s_j,r_j]\|<\sqrt{b}$. 
Then Hadamard's inequality, see \cite[Theorem 16.6]{vzG}, implies that
$$ \det\Big(\begin{bmatrix} s_i & r_i \\ s_j & r_j \\ \end{bmatrix}\Big)
\leq \|[s_i,r_i]\| \cdot \|[s_j,r_j]\| < b .$$
Since $|\det(L_{a,b})|=b$ divides 
$\det\Big(\begin{bmatrix} s_i & r_i \\ s_j & r_j \\ \end{bmatrix}\Big)$
this entails $\langle[s_i,r_i]\rangle_\Q=\langle[s_j,r_j]\rangle_\Q$,
and hence $i=j$ by Lemma \ref{lattice}.
\qed\end{proof}

A comparison of the above treatment with the references already mentioned
seems to be in order: The statement of Proposition \ref{minimum}(a) is 
roughly equivalent to \cite[Theorem]{dixon} and \cite[Theorem 1]{mon}, 
respectively. Alone, the proof given in \cite{dixon} appears to be too
concise, and provides a slightly worse bound for $b$ to be large enough.
And \cite[Theorem 1]{mon} is attributed in turn to \cite{davguywan}, while
for a proof the reader is referred to \cite{vzG}. Unfortunately,
\cite[Theorem 5.26]{vzG} is not immediately conclusive for the 
statements under consideration here.

The main difference between the above-mentioned approaches and ours is 
the break condition used to actually determine the index $i$ 
referred to in Proposition \ref{minimum}(a): In \cite{davguywan,dixon,vzG} 
a bound on the residues $r_i$ is used, while in \cite[Section 3]{mon} the
quotients $q_i$ are considered instead (yielding a randomized algorithm).
In contrast, in our decisive Proposition \ref{minimum}(b) we are using the
minimum of the lattice $L_{a,b}$, which hence treats both the $r_i$ and $s_i$
(in other words the the unknown numbers $y$ and $x$) on a ``symmetric'' 
footing. To our knowledge, this point of view is new, its algorithmic 
relevance being explained below.

\absT{Recovering rational numbers}\label{ratrecog}
We are now prepared to describe our first computational task,
which will appear both in computations over $\Z$ in Section \ref{padicdec}, 
and over the polynomial ring $\Q[X]$ in Section \ref{polrecog}:

Let $x\in\N$ and $0\neq y\in\Z$ such that $\gcd(x,y)=1$.
Assume we are given $a,b\in\N$ such that $\gcd(x,b)=1$ and 
$y\equiv ax\pmod{b}$; 
note that since $x$ is invertible modulo $b$
we may write $\frac{y}{x}\equiv a\pmod{b}$ instead, 
which we will feel free to do if convenient.
Now, if $b$ is large enough compared to $x$ and $|y|$, the task
is to recover $\frac{y}{x}\in\Q$ from its congruence class $a\pmod{b}$.

In view of Proposition \ref{minimum}(b), this is straightforward:
Assuming that $x^2+y^2<b$, the $\Z$-lattice 
$L_{a,b}=\langle[1,a],[0,b]\rangle_\Z\subseteq\Z^2$ 
has precisely two shortest non-zero elements, namely the primitive
elements $\pm [x,y]$. In other words, the rational number
$\frac{y}{x}\in\Q$ can be found by computing a 
shortest non-zero element of $L_{a,b}$. 
This in turn can be done algorithmically by the Gau{\ss} reduction
algorithm for $\Z$-lattices of rank $2$, see \cite[Algorithm 1.3.14]{cohen}. 
Moreover, compared to the general case, 
for the particular lattice $L_{a,b}$ we have a better break condition:
We may stop early as soon as we have found an element $[x,y]\in L_{a,b}$
such that $x^2+y^2<b$. If then $[x,y]$ is primitive, the rational number 
$\frac{y}{x}$ fulfills all assumptions made, where of course 
its correctness has to be verified independently.
Otherwise, if $[x,y]$ is not primitive, or the shortest element 
$[x',y']\in L_{a,b}$ found fulfills $x'^2+y'^2\geq b$, then we report failure.
Thus, in practice, we choose $b$ small, and rerun the above algorithm
with $b$ increasing, until we find a valid candidate passing independent
verification.

At this stage, we should point out the algorithmic advantage
of our approach, compared to the other ones mentioned: 
The latter refer to the convergents of continued fraction expansions, 
and thus to the full sequence of non-negative residues
of the extended Euclidean algorithm. In contrast,
the Gau{\ss} reduction algorithm to find a lattice minimum proceeds 
by iterated pair reduction, starting with the pair $[0,b]$ and $[1,a]$.
Although this is essentially equivalent to running the extended
Euclidean algorithm on $a$ and $b$, here we are allowed to use best 
approximation. This amounts to using numerically smallest residues,
instead of non-negative ones as was necessary in the context of
continued fraction expansions. 
Although we have not carried out a detailed comparison,
it is well-known that this saves a non-negligible
amount of quotient and remainder steps.
\end{abs}

\absT{Finding linear combinations}\label{padicdec}
We are now going to describe \emph{the} basic task we are faced with
in order to be able to do computational linear algebra over $\Z$.
To do so, we of course avoid the Gau{\ss} algorithm over $\Q$,
but we also do not refer to pure ``lattice algorithms'', as they are
called in \cite[Section 2.1]{cohen}, for example those to compute 
Hermite normal forms or reduced lattice bases described in 
\cite[Section 2.4--2.7]{cohen}.
Instead, we use a modular technique, which is a keystone to make 
use of the ideas of the \MA{} in the framework of the \IMA. 
To our knowledge, this has only been discussed very briefly 
in the literature, for example in \cite{dixon,parker2}.
Moreover, our approach differs from those cited, at least in detail;
in particular, \cite{dixon} only allows for regular square matrices.

To describe the computational task, we again need some preparations first:
Given a (rectangular) matrix $A\in\Z^{m\times n}$, with $\Q$-linearly 
independent rows $w_1,\ldots,w_m\in\Z^n$, where $m,n\in\N$, let 
$$ L:=\langle w_1,\ldots,w_m\rangle_\Z\leq\Z^n $$ 
be the $\Z$-lattice spanned by the rows of $A$, 
and let $L\leq\widehat{L}\leq\Z^n$ be its 
\emph{pure closure} in $\Z^n$, that is the smallest pure $\Z$-sublattice
of $\Z^n$ containing $L$. Then the index $\det(L):=[\widehat{L}\colon L]$
is finite; of course, if $m=n$ then we have $\det(L)=|\det(A)|$.
Thus for any vector $v\in\Z^n$, we have $v\in\widehat{L}$ 
if and only of there is $a\in\N$ such that $av\in L$;
in this case, if $a$ is chosen minimal then it divides $\det(L)$.

Now, given $v\in\Z^n$, the task is to decide whether or not
$v\in\widehat{L}$, and if this is the case to compute
$a_1,\ldots,a_m\in\Z$ and $a\in\N$ such that $\gcd(a,a_1,\ldots,a_m)=1$ and 
$$ v=\frac{1}{a}\cdot\sum_{j=1}^m a_j w_j
=\frac{1}{a}\cdot [a_1,\ldots,a_m] \cdot A ;$$
in this case $a$ and the $a_i$ are uniquely determined.
\end{abs}

\subT{The $p$-adic decomposition algorithm}
To do so, we choose a (large) prime $p$. Then reduction modulo $p$ yields 
the matrix $\overline{A}\in\F_p^{m\times n}$ over the prime field $\F_p$.
We assume that the rows $\overline{w}_1,\ldots,\overline{w}_m\in\F_p^n$ 
of $\overline{A}$ are $\F_p$-linearly independent as well;
otherwise we choose another prime $p$. 
By the structure theory of finitely generated modules over 
principal ideal domains, this condition is equivalent to saying
$\overline{\widehat{L}}=\overline{L}$, which in turn is equivalent to 
$p$ not dividing $\det(L)$. In particular, the independence condition 
on $\overline{w}_1,\ldots,\overline{w}_m\in\F_p^n$ is fulfilled for all but 
finitely many primes $p$.

Thus we have $v\in\widehat{L}$ if and only if 
$\overline{v}\in\overline{L}
=\langle\overline{w}_1,\ldots,\overline{w}_m\rangle_{\F_p}$,
solving the decision problem. Furthermore, if $v\in\widehat{L}$ 
then set $v_0:=v$, and for $d\in\N_0$ 
proceed successively as follows: Since $v_d\in\widehat{L}$, there
are $[a_{d,1},\ldots,a_{d,m}]\in\Z^m$ such that 
$-\frac{p}{2}<a_{d,j}\leq\frac{p}{2}$ for all $1\leq j\leq m$, and 
$$ \overline{v}_d
=\sum_{j=1}^m \overline{a}_{d,1}\overline{w}_j
=[\overline{a}_{d,1},\ldots,\overline{a}_{d,m}]
   \cdot\overline{A}\in\F_p^n .$$
Then we let
$$ v_{d+1}:=\frac{1}{p}\cdot
\Big(v_d-[a_{d,1},\ldots,a_{d,m}]\cdot A\Big)\in\Z^n .$$
Hence we have $v_{d+1}\in\widehat{L}$ as well, and we may recurse. 
This yields
$$ v \equiv
\Big(\sum_{i=0}^d p^i\cdot [a_{i,1},\ldots,a_{i,m}]\Big)\cdot A \equiv 
\Big[\sum_{i=0}^d p^i a_{i,1},\ldots,\sum_{i=0}^d p^i a_{i,m}\Big]\cdot A 
    \pmod{p^{d+1}\Z^n} ,$$
or equivalently
$$ \frac{a_j}{a}\equiv\sum_{i=0}^d p^i a_{i,j}\pmod{p^{d+1}}, 
\quad\textrm{for all}\quad 1\leq j\leq m .$$

Thus, if $v\in L$, or equivalently $a=1$, then 
since $-\frac{p^{d+1}}{2}<\sum_{i=0}^d p^i a_{i,j}\leq\frac{p^{d+1}}{2}$
there is some $d\in\N_0$ such that $v_{d+1}=0$, implying that
$a_j=\sum_{i=0}^d p^i a_{i,j}$, for all $1\leq j\leq m$,
without further independent verification necessary.
Otherwise, if $v\in\widehat{L}\setminus L$, then applying 
rational number recovery for some $d\in\N_0$ large enough,
see Section \ref{ratrecog},
reveals the vector $\frac{1}{a}\cdot [a_1,\ldots,a_m]\in\Q^m$;
note that under the assumptions made $p$ does not divide $a$.
In the latter case correctness is independently verified 
by computing $[a_1,\ldots,a_m]\cdot A\in\Z^n$ and checking whether
it equals $av\in\Z^n$.
\end{sub}

\subT{Modular computations}
In practice, to check $\overline{w}_1,\ldots,\overline{w}_m\in\F_p^n$
for $\F_p$-linear independence, and to compute the vectors 
$[\overline{a}_{d,1},\ldots,\overline{a}_{d,m}]\in\F_p^m$ 
we use ideas taken from the \MA{}. In particular, in order to 
keep the depth $d$ needed smallish, but still to be able to make
efficient use of fast arithmetic over small finite prime fields, 
we choose the prime $p$ amongst the largest primes smaller than $2^8=256$.
(In our application we for example use $p=251$ as the default prime.)
\end{sub}

\absT{Nullspace}\label{intnullspace}
In the framework of the \IMA{} there is a general method to 
compute a $\Z$-basis of the row kernel of a matrix with entries 
in $\Z$, see \cite{parker2}. But in view of the application
to row kernels of matrices over $\Q[X]$ in Section \ref{polnullspace},
here we only deal with the following restricted nullspace problem:  

Given a matrix $A\in\Q^{m\times n}$, where $m,n\in\N$,
such that $\dim_{\Q}(\ker(A))=1$,
where $\ker(A)$ denotes the row kernel of $A$, 
compute a primitive vector $v\in\Z^m$ such that 
$\ker(A)=\langle v\rangle_{\Q}$;
then $v$ is unique up to sign.

To do so, by going over to a suitable $\Q$-multiple
we may assume that $A\in\Z^{m\times n}$. Let $w_1,\ldots,w_m\in\Z^n$
be the rows of $A$. We may assume that $w_1\neq 0$, since otherwise
we trivially set $v:=[1,0,\ldots,0]\in\Z^m$.
Then for $2\leq i\leq m$ we successively check,
using the $p$-adic decomposition algorithm in Section \ref{padicdec},
whether or not $w_i\in\langle w_1,\ldots,w_{i-1}\rangle_{\Q}$.
If this is not the case, that is $\{w_1\ldots,w_i\}$ is $\Q$-linearly
independent, then if $\overline{w}_1,\ldots,\overline{w}_i\in\F_p^n$
turns out to be $\F_p$-linearly independent we increment $i$, while 
otherwise we return failure in order to choose another prime $p$.
If $\{w_1\ldots,w_i\}$ is $\Q$-linearly dependent, then the
$p$-adic decomposition algorithm returns 
$a_1,\ldots,a_{i-1}\in\Z$ and $a\in\N$ such that
$\gcd(a,a_1,\ldots,a_{i-1})=1$ and 
$w_i=\frac{1}{a}\cdot\sum_{j=1}^{i-1} a_j w_j$. Thus 
$v:=[a_1,\ldots,a_{i-1},-a,0,\ldots,0]\in\ker(A)\leq\Z^m$ is primitive.
\end{abs}

\absT{Inverse}\label{intinverse}
Matrix inversion over $\Q$, from the point of view of reducing to 
computations over $\Z$ as much as possible, can be formulated as
the following task:

Given a matrix $A\in\Q^{n\times n}$, where $n\in\N$, 
such that $\det(A)\neq 0$, compute $B\in\Z^{n\times n}$ and $c\in\N$,
such that $A^{-1}=\frac{1}{c}\cdot B\in\Q^{n\times n}$ 
and the overall greatest common divisor $\gcd(B,c)$ of 
the entries of $B$ and $c$ equals $\gcd(B,c)=1$; then $(B,c)$ is unique.

To do so, by going over to a suitable $\Q$-multiple 
we may assume that $A\in\Z^{n\times n}$. 
Then the equation $BA=c\cdot E_n$, where $E_n$ 
denotes the identity matrix, implies that $\gcd(B)$ divides $c$,
and hence $B$ is necessarily primitive. 
Solving the equations $\mathcal X A=E_n$, for the unknown matrix 
$\mathcal X\in\Q^{n\times n}$, amounts to writing the 
rows of the identity matrix as $\Q$-linear combinations
of the rows of $A$, which is done using the $p$-adic decomposition
algorithm in Section \ref{padicdec}; recall that the rows of $A$ indeed are 
assumed to be $\Q$-linearly independent. 
\end{abs}

\absT{The exponent of a matrix}\label{intexp}
Given a square matrix $A\in\Z^{n\times n}$ such that $\det(A)\neq 0$
as above, the number $c\in\N$ found in the expression 
$A^{-1}=\frac{1}{c}\cdot B$, where $B\in\Z^{n\times n}$ is 
chosen to be primitive, turns out to have another interpretation:
 
Let $\textrm{im}(A)\leq\Z^n$ be the $\Z$-span of the rows of $A$.
By the structure theory of finitely generated modules over principal
ideal domains, the annihilator of the $\Z$-module $\Z^n/\textrm{im}(A)$ 
is a non-zero ideal of $\Z$, the positive generator $\exp(A)$ of which
is called the \emph{exponent} of $A$. 
Moreover, $\exp(A)$ divides $\det(A)$, which in turn divides some 
power of $\exp(A)$. Thus the prime divisors of $\exp(A)$ are
precisely the primes $p\in\Z$ such that $\overline{A}\in\F_p^{n\times n}$ 
is not invertible.

Now, actually $\exp(A)$ and $c$ coincide:
From $BA=c\cdot E_n$ we conclude that 
$(c\Z)^n\leq\textrm{im}(A)$, hence $\exp(A)$ divides $c$;
conversely, since $(\exp(A)\cdot\Z)^n\leq\textrm{im}(A)$
there is $B'\in\Z^{n\times n}$ such that $B'A=\exp(A)\cdot E_n$,
implying that $\exp(A)\cdot B=c\cdot B'$, which by the primitivity of 
$B$ shows that $c$ divides $\exp(A)$.
In other words, computing the inverse of $A$ as described in Section
\ref{intinverse} also yields a method to compute $\exp(A)$.
\end{abs}
  
\section{Computing with polynomials}\label{polcomp}

Having the necessary pieces of linear algebra over the integers in place,
in this section we describe computational aspects of single polynomials,
before we turn to linear algebra over polynomials rings in 
Section \ref{matcomp}.

\absT{Polynomial arithmetic}
As our general strategy is to use linear algebra over $\Z$ or $\Q$
to do linear algebra over $\Z[X]$ or $\Q[X]$, for all arithmetically
heavy computations we recurse to $\Z$ or $\Q$. Consequently, for the
remaining pieces of explicit computation in $\Z[X]$ or $\Q[X]$ we may
use a simple straightforward approach:

We use our own standard arithmetic 
for polynomials over $\Z$ or $\Q$, where a polynomial 
$0\neq f=\sum_{i=0}^d z_iX^i\in\Q[X]$ is just represented by its coefficient
list $[z_0,\ldots,z_d]\in\Q^{d+1}$ of length $d+1$, where $d=\deg(f)$. 
Thus we avoid structural overhead as much as possible, and may
use directly the facilities to handle row vectors provided by \GAP{}.
But we would like to stress that this is just tailored for our 
aim of doing linear algebra over polynomial rings,
and not intended to become a new general-purpose polynomial arithmetic.
For example, we are not providing asymptotically fast multiplication,
as is for example described in \cite[Section 8.3]{vzG}.

In particular, we only rarely need to compute polynomial greatest common 
divisors. Hence we avoid sophisticated (modular) techniques, as are 
for example described and compared in \cite[Chapter 6]{vzG}, 
but we are content with a simple variant of the Euclidean algorithm:
Assuming that the operands have integral coefficients, by going over to 
$\Q$-multiples if necessary, in order to avoid coefficient explosion
we just use denominator-free pseudo-division as described in
\cite[Algorithm 3.1.2]{cohen}, and Collins's sub-resultant algorithm
given in \cite[Algorithm 3.3.1]{cohen}, albeit the latter without 
intermediate primitivisation. 

On the other hand, we very often have to evaluate polynomials 
at various places, where
our strategy is to use as few of these specializations as possible,
so that evaluation at distinct places is done step by step.
Thus we are not in a position to use multi-point evaluation techniques, 
as are for example described in \cite[Section 10.1]{vzG}. 
Hence we are just using the Horner scheme, which under these circumstances
is well-known to need the optimal number of multiplications.

We now describe the special tasks needed to be solved in our approach:
\end{abs}

\absT{Recovering polynomials}\label{polrecog}
The aim is to recover a polynomial with rational coefficients,
which we are able to evaluate at arbitrary integral places, from as
few such evaluations (at ``small'' places) as possible. More precisely:

Let $0\neq f:=\sum_{i=0}^d z_iX^i\in\Q[X]$ be a polynomial of
degree $d=\deg(f)\in\N_0$, having coefficients $z_i=\frac{y_i}{x_i}\in\Q$,
where $x_i\in\N$ and $y_i\in\Z$ such that $\gcd(x_i,y_i)=1$. 
Then the task is to find pairwise coprime places 
$b_1,\ldots,b_k\in\Z\setminus\{0,\pm 1\}$, for some (small) $k\in\N$,
such that the degree $d$ and the coefficients $z_0,\ldots,z_d$ of $f$ can
be computed from the values $f(b_1),\ldots,f(b_k)\in\Q$ alone.
Note that, in particular, we do not assume that $k>d$, so that
polynomial interpolation is not applicable. 
(Actually, in our application we often enough have $k\ll d$, where for
example $k\sim 5$, but $d\lesssim 200$.)

To this end, let $b:=\prod_{j=1}^k |b_j|\in\N$, and assume that we have 
$\gcd(x_i,b)=1$ and $x_i^2+y_i^2<b$ for all $0\leq i\leq d$.
Hence the congruence classes $z_i\equiv \frac{y_i}{x_i}\pmod{b_j}$ 
and $f(b_j)\pmod{b_j}$ are well-defined, and for the constant
coefficient of $f$ we get
$$ z_0\equiv\sum_{i=0}^d z_ib_j^i\equiv f(b_j)\pmod{b_j},
\quad\textrm{for}\quad
1\leq j\leq k .$$  
Thus by the Chinese Remainder Theorem, 
see for example \cite[Theorem 1.3.9]{cohen}, there is a unique congruence
class $a\pmod{b}$, where $a\in\Z$, such that $a\equiv z_0\pmod{b}$.
To compute $a\in\Z$, we let $a_j\in\Z$ such that
$$ f(b_j)\equiv a_j\pmod{b_j},
\quad\textrm{for}\quad 1\leq j\leq k .$$  
An application of Chinese remainder lifting in $\Z$ to the congruence classes 
$a_1\pmod{b_1},\ldots,a_k\pmod{b_k}$ yields the congruence class $a\pmod{b}$, 
and by our choice of $b$ applying rational number recovery as described 
in Section \ref{ratrecog} reveals $z_0\in\Q$.
Now we recurse to $\widetilde{f}:=\frac{f-z_0}{X}\in\Q[X]$, 
whose value at the place $b_j$ can of course be determined directly
from $f(b_j)$ as $\widetilde{f}(b_j)=\frac{f(b_j)-z_0}{b_j}\in\Q$.
\end{abs}

\subT{Chinese remainder lifting} 
Hence, apart from rational number recovery, the key computational
task to be solved is to perform Chinese remainder lifting in $\Z$:

We are using the straightforward approach based on the extended 
Euclidean algorithm, as is described in \cite[Section 1.3.3]{cohen}. 
Since we are computing many lifts with respect to the same places
$b_1,\ldots,b_k$, we make use of a precomputation step, 
as in \cite[Algorithm 1.3.11]{cohen}. But, since again for reasons of
time and memory efficiency we are choosing small places $b_j$, 
the specially tailored approach in \cite[Algorithm 1.3.11]{cohen} 
to keep the intermediate numbers occurring small, at the expense of
needing more multiplications, does not pay off as experiments show.
Moreover, as we are computing the values $f(b_j)$ 
for $1\leq j\leq k$ step by step, where even the number $k$ of places
is not determined in advance, we cannot take advantage of 
fast Chinese remainder lifting techniques, as are described for example
in \cite[Section 10.3]{vzG}, either.

Our strategy is to rerun the above algorithm with $k$ increasing,
choosing small integral $2\leq b_1<b_2<\cdots<b_k$, 
and to discard quickly erroneous guesses by an independent verification,
until the correct answer passing the verification is found.
By the above discussion, this happens after finitely many iterations.
Before that, if $b=|\prod_{j=1}^k b_j|$ is too small, or not coprime 
to all the denominators $x_i$, the Chinese remainder lifting process
does not terminate, or it terminates with a wrong guess. 
To catch the first case, we impose a degree bound, and stop the 
lifting process with a failure message if it is exceeded, in order 
to increment $k$. (In our application, $200$ turned out to be a
suitable degree bound in all cases.)

To catch the second case, we only allow for denominators $x_i$ dividing an
imposed bound. This is justified, since rational number recovery
as described in Section \ref{ratrecog} is a trade-off between finding the 
numerator $y$ and the denominator $x$ of the rational number 
$\frac{y}{x}$ to be reconstructed: In practice, we typically
encounter small denominators $x$ and large numerators $y$,
which escape the Gau{\ss} reduction algorithm if $b$ is chosen too small, 
since then the latter tends to return a larger denominator $x'>x$ 
and a smaller numerator $|y'|<|y|$. 
(In our application, denominator bounds such as small $2$-powers,
or $12$, or $20$ turned out to be sufficient in all cases.)
\end{sub}

\absT{Degree detection}\label{degdetect}
We keep the setting of Section \ref{polrecog}. The technique to be described 
now has arisen out of an attempt to determine the degree of $f$
without determining its coefficients. Actually, it deals with the
following more general situation (whose relevance for our computations
will be explained in Section \ref{catchproj} below):

Assume that instead of the values $f(b_1),\ldots,f(b_k)$ we are only 
able to compute ``rescaled values'' $a_1 f(b_1),\ldots,a_k f(b_k)\in\Q$,
with scalar factors $a_j\in\Q$ such that $a_j>0$, which are only known to
come from a finite pool $\mathcal R$ of positive rational numbers associated
with $f$. 
Thus the task now becomes to find $k\in\N$ and coprime places 
$b_1,\ldots,b_k\in\Z\setminus\{0,\pm 1\}$ as above, allowing to 
determine $f$ up to some positive rational scalar multiple, that is to find
$af\in\Q[X]$, for some $a\in\Q$ such that $a>0$; note that 
this also determines all the quotients $\frac{a_j}{a}$.

To this end, we let $\alpha_1,\ldots,\alpha_d\in\C$ be the complex roots
of $f$, and set $\mu:=\max\{0,|\alpha_1|,\ldots,|\alpha_d|\}$.
Moreover, since $\mathcal R$ is a finite set, we have 
$$\delta:=\min\{|\ln(a')-\ln(a)|\in\R;a,a'\in\mathcal R,\, a\neq a'\}>0 .$$
Now, let $k\geq 2$, and for the places $b_1,\ldots,b_k$ 
we additionally assume that 
$$ (1+2d)\cdot\mu<b_1<\cdots<b_k
\quad\textrm{and}\quad 
\ln(b_k)-\ln(b_1)<\delta ;$$
hence, in particular, the $f(b_j)$ are non-zero and have the same sign.
The necessity of these choices will become clear below. But this forces
us to show that for all $k\geq 2$ and all
$x>0$ and $\delta>0$ there actually exist pairwise coprime integers 
$b_1<\cdots<b_k$ such that $x<b_1$ and $\ln\big(\frac{b_k}{b_1}\big)<\delta$.
Indeed, we are going to show that the latter can always be chosen
to be primes (where the mere existence proof to follow is impractical,
but in practice considering small primes
works well, see Example \ref{degdetectex}):

Let $p_0<p_1<\cdots$ be the sequence of all primes 
exceeding $x$, and assume to the contrary that for all $k$-subsets thereof,
$q_1<\cdots<q_k$ say, we have $\ln\big(\frac{q_k}{q_1}\big)\geq\delta$. 
Then we have $p_{k-1}\geq e^\delta\cdot p_0$, and thus 
$p_{j(k-1)}\geq e^{j\delta}\cdot p_0$, for all $j\in\N$.
Using the prime number function
$\pi(x):=|\{p\in\N;p\textrm{ prime},p\leq x\}|$ this implies
$$ \pi(e^{j\delta}\cdot p_0)\leq\pi(p_0)+j(k-1). $$
From this we get 
$$ \lim_{j\rightarrow\infty}
   \frac{\pi(e^{j\delta}\cdot p_0)\cdot\ln(e^{j\delta}\cdot p_0)}
        {e^{j\delta}\cdot p_0}
\leq\lim_{j\rightarrow\infty}
   \frac{\big(\pi(p_0)+j(k-1)\big)\cdot\big(j\delta+\ln(p_0)\big)}
        {e^{j\delta}\cdot p_0}
=0 ,$$
contradicting the Prime Number Theorem, 
see \cite[Section 1.8, Theorem 6]{hw}, saying that 
$\lim_{x\rightarrow\infty}\frac{\pi(x)\cdot\ln(x)}{x}=1$.

\end{abs}

\subT{Growth behavior of polynomials}
We now consider the growth behavior of the polynomial $f$.
For $x>\mu$ we have
$$ \frac{\partial}{\partial x}\big(f(x)\big) 
= z_d\cdot\frac{\partial}{\partial x}\big(\prod_{r=1}^d(x-\alpha_r)\big)
= f(x)\cdot\sum_{r=1}^d\frac{1}{x-\alpha_r} ,$$
implying
$$ \frac{\partial}{\partial x}\big(\ln(f(x))\big)
= \frac{\partial}{\partial x}\big(f(x)\big)\cdot\frac{1}{f(x)}
= \sum_{r=1}^d\frac{1}{x-\alpha_r} .$$ 
Thus, for $1\leq i<j\leq k$, by the mean value theorem for derivatives
there is $b_i<\beta<b_j$ such that
$$ \frac{\ln(f(b_j))-\ln(f(b_i))}{\ln(b_j)-\ln(b_i)} 
= \sum_{r=1}^d\frac{\beta}{\beta-\alpha_r} .$$
Since by assumption 
$b_i>(1+2d)\cdot\mu\geq(1+2d)\cdot|\alpha_r|$, we have
$$ \Big|\frac{\beta}{\beta-\alpha_r}-1\Big| 
=\Big|\frac{\alpha_r}{\beta-\alpha_r}\Big| 
\leq \frac{|\alpha_r|}{\beta-|\alpha_r|} 
<\frac{|\alpha_r|}{(1+2d)\cdot|\alpha_r|-|\alpha_r|} 
\leq\frac{1}{2d} $$ 
for all $1\leq r\leq d$. 
All differences $\beta-\alpha_r\in\C$ having positive real parts, we get
$$ d<\frac{\ln(f(b_j))-\ln(f(b_i))}{\ln(b_j)-\ln(b_i)}<d+\frac{1}{2} .$$
Moreover, by assumption we have
$0<\ln(b_j)-\ln(b_i)<\delta\leq|\ln(a_j)-\ln(a_i)|$, hence
$$ \Big|\frac{\ln(a_j)-\ln(a_i)}{\ln(b_j)-\ln(b_i)}\Big|>1 .$$ 
Now, letting $\lfloor x\rceil:=\lfloor x+\frac{1}{2}\rfloor\in\Z$
denote the integer nearest to $x\in\R$, we set
$$ d_{ij} := \Big\lfloor
\frac{\ln(a_jf(b_j))-\ln(a_if(b_i))}{\ln(b_j)-\ln(b_i)}
\Big\rceil=\Big\lfloor
\frac{\ln(f(b_j))-\ln(f(b_i))}{\ln(b_j)-\ln(b_i)}
+\frac{\ln(a_j)-\ln(a_i)}{\ln(b_j)-\ln(b_i)}
\Big\rceil $$
for all $1\leq i,j\leq k$ such that $i\neq j$; note that $d_{ij}=d_{ji}$.
Hence from the above estimates we infer that $d_{ij}=d$ 
if and only if $a_i=a_j$. In particular, all these numbers $d_{ij}$ 
coincide if and only if $a_1=\cdots=a_k$, hence in this case immediately
determining $d$.
\end{sub}

\subT{Combinatorial translation}
Thus our task can now be rephrased in combinatorial terms as follows: 
For $c\in\Z$ let $\Gamma_{d+c}$ be the undirected graph
on the vertex set $\{1,\ldots,k\}$, whose edges are the
$2$-subsets $\{i,j\}\subseteq\{1,\ldots,k\}$ such that $d_{ij}=d+c$.

Then by the above discussion the connected components of $\Gamma_d$
are complete graphs, whose vertex sets coincide with the sets of 
$j\in\{1,\ldots,k\}$ such that the associated scalars $a_j$ assume one
and the same value. 
On the other hand, if $\Gamma_{d+c}$, for some $c\neq 0$,
has a complete connected component with $r\geq 2$ vertices 
$b_{j_1}<\cdots<b_{j_r}$, 
then for all $i,j\in\{j_1,\ldots,j_r\}$ such that $i<j$ we have
$$ c-1<\Big|\frac{\ln(a_j)-\ln(a_i)}{\ln(b_j)-\ln(b_i)}\Big|<c+\frac{1}{2} .$$
Thus we infer that the sequence $a_{j_1},\ldots,a_{j_r}$
is strictly increasing if $c>0$, and strictly decreasing if $c<0$.
In particular this implies that $r\leq|\mathcal R|$.
In other words, as soon as we find a complete connected component
of a graph $\Gamma_{d+c}$ having more than $|\mathcal R|$ elements, 
then we may conclude that $c=0$, and we have determined $d$. Moreover,
if $k>|\mathcal R|^2$ than this case actually happens.

Our algorithm to determine the degree $d$ of $f$, and $af$ for some 
$a>0$, is now straightforward:
Again our strategy is to increase $k$ step by step, and to 
choose places $2\leq b_1<b_2<\cdots<b_k$ such that $b_1$ is growing 
and $\ln(b_k)-\ln(b_1)$ tends to zero. Having made a choice, we compute
the numbers $d_{ij}\in\Z$ for all $1\leq i<j\leq k$; 
note that here we do not see a way to avoid using non-exact floating point 
arithmetic (to evaluate logarithms), while everywhere else we are 
computing exactly. For all numbers $d'\in\Z$ thus occurring
we then determine the graph $\Gamma_{d'}$. Amongst all the graphs 
found we choose one, again $\Gamma_{d'}$ say, having a complete connected 
component of maximal cardinality, with vertex set
$\mathcal J\subseteq\{1,\ldots,k\}$ say. Then we run polynomial recovery,
see Section \ref{polrecog}, using the places $\{b_j;j\in\mathcal J\}$
and the values $\{a_jf(b_j);j\in\mathcal J\}$, with degree bound $d'$. 
\end{sub}

\absT{An example}\label{degdetectex}
Here is an example to illustrate the above process.
(It is a modified version of an example which actually 
occurred in our application.)
Assume as places $b_j$, for $1\leq j\leq k=13$, we have chosen the rational
primes between $29$ and $79$, and evaluating the unknown polynomial 
$f$ has resulted in the list of values $a_jf(b_j)$ given in 
Table \ref{degdetecttbl}; the scalars $a_j$ are of course not known either.

Then it turns out that the numbers $d'\in\Z$, where $1\leq i<j\leq 13$, 
come from an $34$-element subset of $\{-27,\ldots,71\}$.
For seven of them the associated graph $\Gamma_{d'}$ 
has a connected component with at least three vertices, but 
only for two of them we find a complete connected component amongst them:
The graph $\Gamma_{7}$ has a complete connected component
consisting of the vertices $\mathcal B_0:=\{47, 61, 79\}$,
while the graph $\Gamma_{13}$ consists of three connected components,
which all are complete, having the vertices 
$$ \mathcal B_1:=\{37, 43, 47, 53, 67, 73\}, \quad 
\mathcal B_2:=\{31, 41, 61, 71\}, \quad
\mathcal B_3:=\{29, 59, 79\} .$$
Running polynomial recovery, see Section \ref{polrecog},
using the places $\mathcal B_0$
fails by exceeding the degree bound. But running it using 
$\mathcal B_1$ yields $af=\sum_{i=0}^{13} z_iX^i\in\Z[X]$, where
$$ [z_0,\ldots,z_{13}]=[1,4,8,11,12,12,12,12,12,12,11,8,4,1] ,$$
while running it using $\mathcal B_2$ and $\mathcal B_3$ yields 
$\frac{1}{5}\cdot af\in\Q[X]$ and $\frac{1}{25}\cdot af\in\Q[X]$, respectively.
Thus we indeed have $d=\deg(f)=13$, and assuming that $a=1$ 
we have determined the scalars
$a_j$, for $1\leq j\leq 13$, as well. Note that the bounds assumed
in Section \ref{polrecog} are fulfilled; and the roots
of $f$ turning out to be complex roots of unity, implying $\mu=1$,
the bounds assumed in Section \ref{degdetect} are fulfilled 
as well.

It should be noted that for the preceding discussion we have chosen $k$
large enough to exhibit the occurrence of the erroneous set $\mathcal B_0$,
for which we indeed observe that the associated scalars $a_j$ are pairwise
distinct. But this also reveals another practical observation, at least for 
polynomials occurring in the applications in Section \ref{matcomp}: 
The scalars $a_j$, here coming from the 
three-element set $\mathcal R=\{1,\frac{1}{5},\frac{1}{25}\}$, 
typically are not uniformly distributed throughout
$\mathcal R$, but the scalar $a_j=1$ occurs much more frequently than
the other ones.

As was already mentioned, in practice we instead increase $k$ step by step.
Then for the smallest $k\geq 3$ such that the graph $\Gamma_{13}$ 
has a complete connected component with at least three vertices, 
that is for $k=6$, we find the set $\mathcal B:=\{37, 43, 47\}$ 
of places, indeed being associated to the case $a_j=1$. Now
polynomial recovery using $\mathcal B$ readily returns $f$; 
note that the bounds assumed in Section \ref{polrecog} are still fulfilled.
\end{abs}

\begin{table}\caption{An example for degree detection}\label{degdetecttbl}
$$ \begin{array}{|r|r|r||r|}
\hline
j & b_j & a_jf(b_j) & a_j \\
\hline
\hline
 1 & 29 & 471132000262895400 & \frac{1}{25} \rule{0em}{1.2em} \\
 2 & 31 & 5556161802048405504 & \frac{1}{5} \rule{0em}{1.2em} \\
 3 & 37 & 271378870503231142344 & 1 \rule{0em}{1.2em}  \\
 4 & 41 & 203982274364082601464 & \frac{1}{5} \rule{0em}{1.2em} \\
 5 & 43 & 1885780898401789278912 & 1 \rule{0em}{1.2em} \\
 6 & 47 & 5946135224244400779264 & 1 \rule{0em}{1.2em} \\
 7 & 53 & 28077873950889396256392 & 1 \rule{0em}{1.2em} \\
 8 & 59 & 4493456499569142283200 & \frac{1}{25} \rule{0em}{1.2em} \\
 9 & 61 & 34577756822169042208584 & \frac{1}{5} \rule{0em}{1.2em} \\
10 & 67 & 581970465933078043504704 & 1 \rule{0em}{1.2em} \\
11 & 71 & 246522309921169431519744 & \frac{1}{5} \rule{0em}{1.2em} \\
12 & 73 & 1766015503219395154436952 & 1 \rule{0em}{1.2em} \\
13 & 79 & 196427398952317706342400 & \frac{1}{25} \rule{0em}{1.2em} \\
\hline
\end{array} $$
\end{table}



\absT{Catching projectivities}\label{catchproj}
We now have to explain where the conditions imposed in 
Section \ref{degdetect} come from: Typically, for example for 
the tasks described in Sections \ref{polnullspace} and \ref{polinverse},  
our aim is to determine a matrix over $\Z[X]$ or $\Q[X]$
by computing various specializations first, that is evaluating at
certain places $b_1,\ldots,b_k$, performing some linear algebra over
$\Z$ or $\Q$, as described in Section \ref{intcomp}, for each of the
specializations, and then lifting back to polynomials as explained in 
Section \ref{polrecog}. But the linear algebra step in between might only be 
unique up to a scalar in $\Q$, which additionally depends on the
particular specialization considered. On the other hand, the matrix  
we are looking for might also only be unique up to a scalar in $\Q(X)$.

Let us now, again, agree on the following convention:
Given $f,g\in\Z[X]$, not both zero, let $\gcd(f,g)\in\Z[X]$ denote 
the polynomial greatest common divisor of $f$ and $g$ with 
positive leading coefficient.  
A vector $0\neq v\in\Q[X]^m$, where $m\in\N$,
is called \emph{primitive}, if actually $v\in\Z[X]^m$, and for the 
greatest common divisor $\gcd(v)$ of its entries we have $\gcd(v)=1$.
Clearly greatest common divisor computations in $\Z$ and in $\Z[X]$
yield a $\Q(X)$-multiple of $v$ which is primitive. 
Similarly, a matrix $A\in\Q[X]^{m\times n}$, where $m,n\in\N$, 
is called \emph{primitive}, if actually $A\in\Z[X]^{m\times n}$, and 
for the greatest common divisor $\gcd(A)$ of its entries we have $\gcd(A)=1$.
\end{abs}

\subT{\bf Specializing primitive vectors}
Hence, in the above context the task is to recover a primitive vector 
$[f_1,\ldots,f_m]\in\Z[X]^m$ not from specializations
$[f_1(b_j),\ldots,f_m(b_j)]\in\Z^m$, for $1\leq j\leq k$,
but from ``rescaled'' versions
$[a_jf_1(b_j),\ldots,a_jf_m(b_j)]\in\Q^m$ instead.
This places us in the setting of Section \ref{degdetect}, 
but it remains to justify the assumption that the scalars $a_j\in\Q$ 
involved indeed come from a finite pool:
\end{sub}

\begin{prop}\label{polgcd}
Let $f_1,\ldots,f_m\in\Z[X]$, where $m\in\N$, such that
$\gcd(f_1,\ldots,f_m)=1\in\Z[X]$.
Then there is a finite set $\mathcal P\subseteq\N$ such that 
for all $b\in\Z$ we have 
$$ \gcd(f_1(b),\ldots,f_m(b))\in\mathcal P .$$ 
\end{prop}

\begin{proof} 
Note first that by assumption $f_1,\ldots,f_m$ do not have 
any common zeroes, so that $\gcd(f_1(b),\ldots,f_m(b))\in\N$
is well-defined for any $b\in\Z$.  
We proceed by induction on $m\in\N$. For $m=1$ we have
$f_1=\pm 1$, and we may let $\mathcal P:=\{\pm 1\}$.
Hence let $m\geq 2$, where we may assume that all the $f_i$,
for $1\leq i\leq m$, are non-constant. Letting 
$g:=\gcd(f_1,\ldots,f_{m-1})\in\Z[X]$ we have $\gcd(g,f_m)=1$
Letting $g_i:=f_i/g\in\Z[X]$ for $1\leq i\leq m-1$,
we have $\gcd(g_1,\ldots,g_{m-1})=1$, thus by induction let 
$\mathcal Q\subseteq\N$ be a set as asserted associated with
$g_1,\ldots,g_{m-1}$. 
Now, given $b\in\Z$, we may write
$$ x:=\gcd(f_1(b),\ldots,f_m(b))
=\gcd(g(b)g_1(b),\ldots,g(b)g_{m-1}(b),f_m(b))  $$
as $x=yz$, where $y=\gcd(g(b),f_m(b))$,
and $z$ divides $\gcd(g_1(b),\ldots,g_{m-1}(b),f_m(b))$.
Hence $z$
divides $\gcd(g_1(b),\ldots,g_{m-1}(b))$, and thus divides an element of 
$\mathcal Q$. 
Moreover, from $\gcd(g,f_m)=1$ we 
infer that the resultant $\rho:=\textrm{res}(g,f_m)\in\Z$ is different
from zero, see \cite[Corollary 6.20]{vzG}, which by
\cite[Corollary 6.21]{vzG} implies that $y=\gcd(g(b),f_m(b))$ divides $\rho$.
Thus the set $\mathcal P$ of all positive divisors of the elements of 
$\rho\mathcal Q:=\{\rho r\in\N;r\in\mathcal Q\}$ is as desired.
\qed\end{proof}

\section{Linear algebra over polynomial rings}\label{matcomp}

As was already mentioned, our general strategy to determine matrices 
over $\Z[X]$ or $\Q[X]$ is to specialize first at integral places,
to apply linear algebra techniques as described in Section \ref{intcomp}
to the matrices over $\Z$ or $\Q$ thus obtained, and subsequently to
recover the polynomial entries in question by the Chinese remainder
lifting technique described in Section \ref{polrecog}, applying degree
detection as described in Section \ref{degdetect} if necessary.
In this section we describe how we can do linear algebra over
$\Z[X]$ or $\Q[X]$ using this approach.

Since we are faced with both sparse and dense matrices,
we keep two corresponding formats for matrices over polynomial rings.
(In our application, representing matrices for $W$-graph representations,
see Definition \ref{wgraph}, are extremely sparse, while
Gram matrices for them, see Remark \ref{edirem1}, typically 
are dense; see also Example \ref{gramex}).
We have conversion and multiplication routines between them,
but whenever it comes to linear algebra computations we always use the 
dense matrix format. From the arithmetical side, we are only using 
standard matrix multiplication, but no asymptotically faster methods,
as are for example indicated in \cite[Section 12.1]{vzG}.

\absT{Nullspace}\label{polnullspace}
We have developed a solution to the following restricted 
nullspace problem only (which is sufficient for our application):

Given a matrix $A\in\Q[X]^{m\times n}$, where $m,n\in\N$, 
such that $\textrm{rk}_{\Q[X]}(\ker(A))=1$, 
the task is to determine a primitive vector $v\in\Z[X]^m$ such that
$\ker(A)=\langle v\rangle_{\Q[X]}$; then the vector $v$ is unique up to sign.

To do so, by going over to a suitable $\Q(X)$-multiple 
we may assume that $A\in\Z[X]^{m\times n}$ is primitive.
Then we specialize the matrix $A$ successively at integral places
$b_1,\ldots,b_k$, yielding matrices $A(b_j)\in\Z^{m\times n}$.
Since the rank condition on $A$ is equivalent to saying that
$\det(A')=0$ for all $(m\times m)$-submatrices $A'$ of $A$, while 
there is an $((m-1)\times(m-1))$-submatrix $A''$ of $A$ such that 
$\det(A'')\neq 0$, we have $\textrm{rk}_{\Z}(\ker(A(b)))\geq 1$
for any $b\in\Z$, and for all but finitely many such $b$ we indeed 
have $\textrm{rk}_{\Z}(\ker(A(b)))=1$.
Thus we may assume that all the chosen specializations 
$A(b_j)$ also fulfill $\textrm{rk}_{\Z}(\ker(A(b_j)))=1$.
Note that this provides an implicit check whether the 
rank condition on $A$ indeed holds.

Hence we are in a position to compute the row kernels 
$\ker(A(b_j))=\langle v_j\rangle_\Z\leq\Z^m$ as described in 
Section \ref{intnullspace},
where the $v_j\in\Z^m$ are primitive, for all $1\leq j\leq k$.
Thus the latter are of the form $v_j=\frac{1}{a_j}\cdot v(b_j)$, 
where $a_j=\gcd(v(b_j))\in\N$, and $v\in\Z[X]^m$ is the desired 
primitive solution vector from above. 
By Proposition \ref{polgcd} we conclude that the scalars $a_j$
involved indeed come from a finite pool only depending on $v$.

Now applying degree detection, see Section \ref{degdetect},
and polynomial recovery, see Section \ref{polrecog},
yields candidate vectors $0\neq\widetilde{v}\in\Q[X]^m$,
which by going over to a suitable $\Q$-multiple can be assumed to 
be primitive. Then the correctness of $\widetilde{v}$ can be independently 
verified by explicitly computing $\widetilde{v}A$ and checking 
whether this is zero.
\end{abs}

\absT{Inverse}\label{polinverse}
Given a matrix $A\in\Q[X]^{n\times n}$,
where $n\in\N$, such that $\det(A)\neq 0$,
the task is to find $B\in\Z[X]^{n\times n}$ and $c\in\Z[X]$,
such that $A^{-1}=\frac{1}{c}\cdot B\in\Q(X)^{n\times n}$ 
and the overall greatest common divisor $\gcd(B,c)\in\Z[X]$ of 
the entries of $B$ and $c$ equals $\gcd(B,c)=1$;
then the pair $(B,c)$ is unique up to sign. 

To do so, by going over to a suitable $\Q$-multiple 
we may assume that $A\in\Z[X]^{n\times n}$.
Thus the equation $BA=c\cdot E_n$
implies that $\gcd(B)$ divides $c$, and hence $B$ is primitive. 
Then we specialize the matrix $A$ successively at integral places
$b_1,\ldots,b_k$, yielding matrices $A(b_j)\in\Z^{n\times n}$.
Since for all but finitely many $b\in\Z$ we have $\det(A(b))\neq 0$, 
we may assume that all the chosen specializations $A(b_j)$ indeed
also fulfill $\det(A(b_j))\neq 0$.
Note that this provides an implicit check whether the 
invertibility condition on $A$ indeed holds.

Hence we are in a position to compute the inverses 
$A(b_j)^{-1}\in\Q^{n\times n}$ as described in Section \ref{intinverse},
yielding $B_j\in\Z^{n\times n}$ and $c_j\in\Z$, such that $B_j$ is primitive
and $A(b_j)^{-1}=\frac{1}{c_j}\cdot B_j$, for all $1\leq j\leq k$.
Thus, if $B\in\Z[X]^{n\times n}$ and $c\in\Z[X]$ are the desired solutions
from above, we infer
$$ B_j=\frac{1}{a_j}\cdot B(b_j)  
\quad\textrm{and}\quad
c_j=\frac{1}{a_j}\cdot c(b_j),
\quad\textrm{where}\quad
a_j:=\gcd(B(b_j),c(b_j))\in\N .$$ 
By Proposition \ref{polgcd} we conclude that the scalars $a_j$
involved indeed come from a finite pool only depending on $B$ and $c$.

Now applying degree detection, see Section \ref{degdetect},
and polynomial recovery, see Section \ref{polrecog}, yields candidate 
solutions $\widetilde{B}\in\Q[X]^{n\times n}$ and $\widetilde{c}\in\Q[X]^n$, 
for which by going over to a suitable $\Q$-multiple we may assume that
$\widetilde{c}\in\Z[X]^n$ and $\widetilde{B}\in\Z[X]^{n\times n}$ 
is primitive. 
Then the correctness of $(\widetilde{B},\widetilde{c})$ can be 
independently verified by explicitly computing $A\widetilde{B}$ and checking
whether it equals $\widetilde{c}\cdot E_n$.
\end{abs}

\absT{The exponent of a matrix}\label{polexp}
In view of the discussion in Section \ref{intexp}, and noting that $\Q[X]$ 
is a principal ideal domain as well, we pursue the analogy between matrix 
inverses over $\Z$ and over $\Q[X]$ still a little further.
Indeed, given a square matrix $A\in\Z[X]^{n\times n}$ such that
$\det(A)\neq 0$ as above, the polynomial $c\in\Z[X]$ 
in the expression $A^{-1}=\frac{1}{c}\cdot B$, where $B\in\Z[X]^{n\times n}$
is chosen primitive, again has another interpretation:

Let the \emph{exponent} $\exp(A)\in\Z[X]$ of $A$ be a primitive generator 
of the annihilator of the $\Q[X]$-module $\Q[X]^n/\textrm{im}(A)$,
where $\textrm{im}(A)\leq\Q[X]^n$ is the $\Q[X]$-span of the rows of $A$;
then $\exp(A)$ is unique up to sign. 
Then, similar to Section \ref{intexp}, we conclude that 
$\exp(A)$ and $c$ are associated in $\Q[X]$, 
and thus the primitivity of $\exp(A)$ yields
$$ c=\gcd(c)\cdot\exp(A)\in\Z[X] .$$

In other words, computing the inverse of $A$ as described in Section
\ref{polinverse} also yields a method to compute the exponent of $A$ 
as $\exp(A)=\frac{1}{\gcd(c)}\cdot c$. Moreover, $c$ governs modular
invertibility of $A$ as follows:
\end{abs}

\begin{prop}\label{polexpprop}
We keep the notation of Section \ref{polexp}.
Let $\{0\}\neq\mathfrak p\lhd\Z[X]$ be a prime ideal, 
let $Q_{\mathfrak p}:=\operatorname{Quot}(\Z[X]/\mathfrak p)$ be 
the field of fractions of the integral domain $\Z[X]/\mathfrak p$,
and let $A_{\mathfrak p}\in(\Z[X]/\mathfrak p)^{n\times n}$
be the matrix obtained from $A$ by reduction modulo $\mathfrak p$. 
Then $A_{\mathfrak p}$ is invertible in $Q_{\mathfrak p}^{n\times n}$ 
if and only if $c\not\in\mathfrak p$.
\end{prop}

\begin{proof}
The prime ideals of $\Z[X]$ being well-understood, 
we are in precisely one of the following cases:
(i) We have $\mathfrak p=(p)$, where $p\in\Z$ is a prime;
then we have $Q_{\mathfrak p}\cong\textrm{Quot}(\F_p[X])=\F_p(X)$,
a rational function field;
(ii) we have $\mathfrak p=(f)$, where $f\in\Z[X]$ is non-constant and
irreducible, hence in particular is primitive; then we have
$Q_{\mathfrak p}\cong\Q[X]/(f)$, an algebraic number field;
(iii) we have $\mathfrak p=(p,f)$, where $p$ and $f$ are as above;
then we have $Q_{\mathfrak p}=\Z[X]/\mathfrak p\cong\F_p[X]/(\overline{f})$, 
a finite field.

Now $A_{\mathfrak p}$ is non-invertible in $Q_{\mathfrak p}^{n\times n}$ 
if and only if $\det(A)\in\mathfrak p$, which holds if and only if there 
is an irreducible divisor of $\det(A)$ being contained in $\mathfrak p$.
Thus is suffices to determine (i) the primes $p\in\Z$, and (ii)
the non-constant irreducible polynomials $f\in\Z[X]$ dividing
$\det(A)$ in $\Z[X]$.

(i) 
From $A^{-1} = \frac{1}{\det(A)}\cdot\textrm{adj}(A)\in\Q(X)^{n\times n}$,
where $\textrm{adj}(A)\in\Z[X]^{n\times n}$ is the adjoint matrix of $A$, 
we infer that $c$ divides $\det(A)$ in $\Z[X]$. Hence any prime $p\in\Z$ 
dividing $\gcd(c)$ also divides $\det(A)$ in $\Z[X]$. 
Conversely, if $p$ does not divide $\gcd(c)$, then $p$-modular reduction
yields $\overline{AB}=\overline{cE_n}\neq 0\in\F_p[X]^{n\times n}$,
hence $\det(\overline{A})\neq 0\in\F_p[X]$.
Hence the primes $p\in\Z$ we are looking for 
are precisely the prime divisors of $\gcd(c)$.

(ii)
This is equivalent to finding the irreducible polynomials in $\Q[X]$
dividing $\det(A)$ in $\Q[X]$.
Again similar to Section \ref{intexp} we conclude that 
the latter are precisely the irreducible polynomials dividing $\exp(A)$.
Hence the polynomials $f\in\Z[X]$ we are looking for are precisely the
non-constant irreducible divisors of $\frac{1}{\gcd(c)}\cdot c$.
\qed\end{proof}

\absT{Product}\label{polproduct}
Given matrices $A\in\Q[X]^{l\times m}$ and $B\in\Q[X]^{m\times n}$,
where $l,m,n\in\N$, the task is to compute their product
$AB\in\Q[X]^{l\times n}$.

This is straightforwardly done: Again, by going over to suitable 
$\Q$-multiples we may assume that $A\in\Z[X]^{l\times m}$ and
$B\in\Z[X]^{m\times n}$.
Then we specialize the matrices $A$ and $B$ successively at integral places
$b_1,\ldots,b_k$, yielding matrices $A(b_j)\in\Z^{l\times m}$ and
$B(b_j)\in\Z^{m\times n}$, whose products $A(b_j)B(b_j)\in\Z^{l\times n}$ 
we compute. Now applying 
polynomial recovery, see Section \ref{polrecog}, yields candidate 
solutions $\widetilde{C}\in\Q[X]^{l\times n}$.
(Note that since no ``rescaling'' takes place here it is not necessary
to apply degree detection.)

As for correctness, there are a few necessary conditions which 
can be used as break conditions in polynomial recovery:
All entries of $\widetilde{C}$ must be polynomials with 
integer coefficients, 
and the degrees of the entries of the input matrices 
yield bounds on the degrees of those of $\widetilde{C}$. 
But these conditions are far from being sufficient, so that, in contrast 
to the tasks in Sections \ref{polnullspace} and \ref{polinverse}, 
here we do not have a general way of independently verifying correctness.
(In our application, as a very efficient break condition
we have used the fact that the entries of $\widetilde{C}$
have to be of a particular form, see Section \ref{final}.)
\end{abs}

\absT{An alternative approach}
The idea of our approach is, essentially, to reduce computations over
$\Q[X]$ to computations over $\Z$, where lifting back to 
polynomials is done in one step by combining
specialization and Chinese remainder lifting. In consequence, we 
almost entirely use arithmetic in characteristic zero (except the use
of a large prime field in the $p$-adic decomposition algorithm in
Section \ref{padicdec}). But it seems to be worth-while to say a few
more words on the following ``two-step'' approach, which was
already mentioned briefly in Sections \ref{sec0} and \ref{wgr1a}:

Assume our aim is to determine a matrix $0\neq A\in\Q[X]^{m\times n}$,
where $m,n\in\N$.
To this end, we choose pairwise distinct places 
$b_1,\ldots,b_k\in\Z$, for some $k\in\N$ such that $k>d$, 
where $d\in\N_0$ is the maximum of the degrees of 
the non-zero entries of $A$. Thus, if we are able to compute the 
specializations $A(b_j)\in\Q^{n\times n}$, for $1\leq j\leq k$,
we may recover the entries of $A$ by polynomial interpolation, 
as for example is described in \cite[Section 10.2]{vzG}. 
In turn, to find the specializations $A(b_j)$
we choose pairwise distinct primes $p_1,\ldots,p_l\in\N$, 
for some $l\in\N$, such that the denominators of all the entries of
$A(b_j)$ are coprime to $p_i$, for all $1\leq j\leq k$ and $1\leq i\leq l$.
Then reduction modulo the chosen primes yields matrices
$A_{p_i}(b_j)\in\F_{p_i}^{m\times n}$. 
Hence, if $\prod_{i=1}^l p_i$ is large enough,
and we are able to compute the modular reductions $A_{p_i}(b_j)$,
for $1\leq i\leq l$, then rational number recovery,
see Section \ref{ratrecog}, reveals the entries of $A(b_j)$.
Hence this reduces finding the matrix $A$ to finding the matrices
$A_{p_i}(b_j)$ over prime fields, for which we in turn may use 
techniques of the \MA{}.

Thus here specialization and Chinese remainder lifting are done 
in two separate steps, aiming at taking advantage of the efficiency of 
computations in prime characteristic. But the
``two-step'' approach has severe disadvantages:
The number $k$ of places to specialize at is at least as large as the
degree of the polynomials in question, hence many more and larger $b_j$ 
than in our approach are needed, increasing time and memory requirements,
presumably drastically. (In our application this means $k\lesssim 200$.)
Moreover, in order to use rational number recovery, the number $l$
of primes used for modular reduction must not be too small, at the 
expense of possibly loosing the very fast arithmetic over small finite 
fields, which otherwise is a major advantage of the \MA{}.

Actually, apart from our own experiences, this kind of approach 
is pursued in \cite{mllt},
and the figures on timings and memory consumption given there
seem to support the above comments. But it should be stressed 
that the emphasis of \cite{mllt} is on parallelizing this kind 
of computations, which we here do not consider at all.
\end{abs}

\section{Computing with representations}\label{repcomp}

As was already mentioned in Section \ref{sec0}, 
in our application we will make use of a suitable variant of the 
``standard basis algorithm'', which was originally used in \cite{parker1} 
for computations over finite fields. In this section we present
the necessary ideas from computational representation theory,
which can be formulated in terms of the following general setting:

\absT{Standard bases}\label{stdbas}
Let $\mathcal A$ be a $K$-algebra, where $K$ is a field, being
generated by the (ordered) set $A_1,\ldots,A_r$, where $r\in\N_0$.
Moreover, let $\mathfrak X\colon\mathcal A\rightarrow K^{n\times n}$
be an absolutely irreducible matrix representation of $\mathcal A$,
where $n\in\N$. Then the task is to find a ``canonical'' $K$-basis of 
the row space $K^n$ with respect to the representation $\fX$,
where we consider right actions, as is common in the computational world. 

To this end, let $A_0\in\mathcal A$ such that $\dim_K(\ker(\fX(A_0)))=1$;
note that whenever $\fX$ is irreducible such an element $A_0$ exists
if and only if $\fX$ is absolutely irreducible.
This leads to the following breadth-first search algorithm;
see also \cite{parker1}:
Choose a seed vector $0\neq u\in\ker(\fX(A_0))$,
let $\mathfrak B:=[u]$ and $\mathfrak T:=[[0,0]]$, and set $i:=1$.
As long as $i$ does not exceed the cardinality of $\mathfrak B$,
let $v$ be the $i$-th element of $\mathfrak B$. Then
for $1\leq j\leq r$ let successively $w:=v\cdot\fX(A_j)$, 
and check whether or not $w\in\langle\mathfrak B\rangle_K$.
If so, then discard $w$; if not,
then append $w$ to $\mathfrak B$, and append $[i,j]$ to $\mathfrak T$.
Having done this for all $j$, increment $i$ and recurse.

Since the growing set $\mathfrak B$ is $K$-linearly independent throughout,
this algorithms terminates after at most $n$ loops. After termination, 
$\langle\mathfrak B\rangle_K$ is a non-zero submodule of the irreducible 
$\mathcal A$-module $K^n$, and thus $\mathfrak B$ indeed is a $K$-basis. 
(Of course, we may terminate early, without any further checking,
as soon as the cardinality of $\mathfrak B$ equals $n$, since from this
point on $\mathfrak B$ would not change anymore anyway.)
The (ordered) set $\mathfrak B$ is called a \emph{standard basis} of $K^n$
with respect to the representation $\fX$, the generators $A_1,\ldots,A_r$, 
and the distinguished element $A_0$, and the ``bookkeeping list'' 
$\mathfrak T$ is called the associated \emph{Schreier tree}.

Strictly speaking, $\mathfrak B$ also depends on the chosen seed vector,
but it is essentially unique in the following sense:
If $0\neq\widetilde{u}\in\ker(\fX(A_0))$ gives rise to the 
standard basis $\widetilde{\mathfrak B}$ with Schreier tree
$\widetilde{\mathfrak T}$, then we have $\widetilde{u}=c\cdot u$, 
for some $0\neq c\in K$, and thus $\widetilde{\mathfrak B}=c\cdot\mathfrak B$ 
and $\widetilde{\mathfrak T}=\mathfrak T$.
Moreover, using the Schreier tree $\mathfrak T=[[i_1,j_j],\ldots,[i_n,j_n]]$,
we may recover $\mathfrak B=[u_1,\ldots,u_n]$, up to a scalar, without
any searching as follows: Choose $0\neq u_1\in\ker(\fX(A_0))$, and for 
$2\leq k\leq n$ let successively $u_k:=u_{i_k}\cdot\fX(A_{j_k})$.
\end{abs}

\subT{In practice}
We are able to run the above standard basis algorithm
in the following particular cases:
If $K$ is a (small) finite field, then this can of course be done using 
ideas from the \MA{}, as is already described in \cite{parker1}.

More important from our point of view is the case $K=\Q$.
Then we may assume that $u\in\Z^n$, and if additionally
$\fX(A_i)\in\Z^{n\times n}$, for all $1\leq i\leq r$, then we have
$\mathfrak B\subseteq\Z^n$, hence the key step in the above algorithm,
to decide whether or not $w\in\langle\mathfrak B\rangle_\Q$, can be done 
using the $p$-adic decomposition algorithm in Section \ref{padicdec},
where whenever $\mathfrak B$ is enlarged we also check whether
its $p$-modular reduction $\overline{\mathfrak B}\subseteq\F_p^n$ 
is $\F_p$-linearly independent; if not, then we return failure in order
to choose another prime $p$. (Note that this is reminiscent of the
strategy in Section \ref{intnullspace}.)
\end{sub} 

\absT{Computing homomorphisms}\label{hom}
We return to the general setting in Section \ref{stdbas}, and let 
$\fX'\colon\mathcal A\rightarrow K^{n\times n}$ be a 
matrix representation of $\mathcal A$, which is equivalent to $\fX$.
Then a standard basis $\mathfrak B'=[v'_1,\ldots,v'_n]$ of $K^n$ 
with respect to the representation $\fX'$ is found by choosing
$0\neq v'_1\in\ker(\fX'(A_0))$ and just applying the Schreier tree
$\mathfrak T=[[i_1,j_j],\ldots,[i_n,j_n]]$ already known from the
standard basis computation for $\fX$ by letting successively
$v'_k:=v'_{i_k}\cdot\fX'(A_{j_k})$, for $2\leq k\leq n$;
note that by assumption we indeed have $\dim_K(\ker(\fX'(A_0)))=1$.

Now let $0\neq C\in K^{n\times n}$ be an $\mathcal A$-homomorphism
from $\fX$ to $\fX'$, that is we have 
$$ \fX(A)\cdot C=C\cdot\fX'(A)
\quad\textrm{for all}\quad A\in\mathcal A ;$$ 
of course, it suffices to require this condition for the generators
$A_1,\ldots,A_r$ only. Since $\fX$ is absolutely irreducible,
it follows that $C\in\GL_n(K)$ and is unique up to a scalar.
Moreover, we have $\ker(\fX(A_0))\cdot C=\ker(\fX'(A_0))$,
and thus going over from the standard bases $\mathfrak B$ and $\mathfrak B'$
with respect to $\fX$ and $\fX'$, respectively,
to the associated invertible matrices $B$ and $B'$ with rows
$v_1,\ldots,v_n\in K^n$ and $v'_1,\ldots,v'_n\in K^n$, respectively,
we get $B\cdot C=B'$, or equivalently 
$$ C=B^{-1}\cdot B'\in\GL_n(K) .$$

Thus to determine $C$ we have to perform the following steps:
find $A_0\in\mathcal A$ such that $\dim_K(\ker(\fX(A_0)))=1$;
compute $\ker(\fX(A_0))\leq K^n$ and $\ker(\fX'(A_0))\leq K^n$;
compute a Schreier tree $\mathfrak T$ with respect to $\fX\cong\fX'$ and
$A_0$; apply the Schreier tree $\mathfrak T$ in order to 
compute standard bases $\mathfrak B$ and $\mathfrak B'$ of $K^n$ 
with respect to $\fX$ and $\fX'$, respectively;
going over to matrices, compute the inverse $B^{-1}\in\GL_n(K)$;
and compute the product $C=B^{-1}\cdot B'\in\GL_n(K)$.
\end{abs}

\subT{In practice}
If $K=\Q(X)$, the nullspaces required can be found as 
described in Section \ref{polnullspace}, where we may assume that 
$v_1$ and $v'_1$ are primitive. Moreover, computing matrix inverses 
and matrix products can be done as described in Sections \ref{polinverse}
and \ref{polproduct}, respectively; by multiplying with a suitable
element of $K$ we may assume that $C$ is primitive as well, then $C$
is unique up to sign. 
Hence for our application it remains to describe how a distinguished 
element and a Schreier tree can be found, and we have to give an efficient
break condition for the algorithm in Section \ref{polproduct}.
\end{sub}

\section{Finding standard bases for $W$-graph representations}\label{bilcomp}

We have now described the necessary infrastructure from 
linear algebra over integral domains, and some relevant general ideas 
how to compute with representations, to proceed to the explicit 
determination of Gram matrices of invariant bilinear forms for
balanced representations of Iwahori--Hecke algebras. We recall 
the setting of Section \ref{wgr1a}, which we keep from now on:

Let $(W,S)$ be a finite Coxeter group, and let $\cH_A\subseteq\cH_K$ 
be the associated generic Iwahori--Hecke algebras with equal parameters over
the ring $A=\Z[v,v^{-1}]$ and the field $K=\Q(v)$, respectively, 
being generated by $\{T_s;s\in S\}$.
Moreover, let $\fX^\lambda\colon\cH_K\rightarrow K^{n\times n}$, 
where $n=d_\lambda$, be a $W$-graph representation associated with
$\lambda\in\Lambda$, and let
$$ (\fX^\lambda)'\colon\cH_K\rightarrow K^{n\times n}\colon 
T_w\mapsto\fX^\lambda(T_{w^{-1}})^{\operatorname{tr}}
\quad\textrm{for all}\quad w\in W .$$

As far as computer implementations are concerned, it is more
convenient and more efficient to work with row vectors instead of
column vectors. Therefore, we will now work throughout with right actions
rather than left actions as in Section \ref{sec:1}.
Our aim is to find a primitive Gram matrix 
$P\in\Z[v]^{n\times n}$ for $\fX^\lambda$, 
that is, using the language of right actions, a primitive matrix such that
$$ \fX^\lambda(T_w)\cdot P = P\cdot(\fX^\lambda)'(T_w)
\quad\textrm{for all}\quad w\in W .$$ 
Thus the task is to find a non-zero $\cH_K$-homomorphism from
$\fX^\lambda$ to $(\fX^\lambda)'$. 
In order to use the approach described in Section \ref{hom},
we proceed as follows, where the basic idea of this strategy 
has already been indicated in \cite[Section 4.3]{gemu}:

\absT{Finding seed vectors}\label{seedvec}
To find a suitable seed vector $u_1\in K^n$ for the standard
basis algorithm with respect to $\fX^\lambda$, we proceed as follows: 

Specializing $v\mapsto 1$ we from $\cH_A$
recover the group algebra $\Q[W]$, and $\fX^\lambda$ corresponds to an 
irreducible representation 
$\mathfrak Y^\lambda\colon\Q[W]\rightarrow\Q^{n\times n}$.
In particular, the index and sign representations of $\cH_K$, given by
$\operatorname{ind}_\cH\colon T_s\mapsto v$ and
$\operatorname{sgn}_\cH\colon T_s\mapsto -v^{-1}$, respectively,
for all $s\in S$,
correspond to the trivial and sign representations of $\Q[W]$, given by
$\operatorname{1}_W\colon s\mapsto 1$ and
$\operatorname{sgn}_W\colon s\mapsto -1$, respectively.

As was observed by Benson and Curtis (see \cite[Section 6.3]{gepf} and the
references there), there is a subset 
$J\subseteq S$ (depending on $\lambda$, and in general not being unique),
such that the restriction of $\mathfrak Y^\lambda$ to the parabolic subgroup
$\widetilde W:=W_J\leq W$ associated with $J$ fulfills
$$ \dim_\Q\big(\Hom_{\Q[\widetilde W]}(
\operatorname{sgn}_{\widetilde W},\mathfrak Y^\lambda)\big)=1 .$$
Note that $J=\emptyset$ and $J=S$ if and only if $\mathfrak Y^\lambda$ 
equals $1_W$ and $\operatorname{sgn}_W$, respectively.  
Letting $\widetilde\cH_K\subseteq\cH_K$ be the parabolic subalgebra  
associated with $J$, this implies
$$ \dim_K\big(\Hom_{\widetilde\cH_K}(
\operatorname{sgn}_{\widetilde\cH},\fX^\lambda)\big)=1 .$$
In other words, we equivalently have
$$\dim_K\Big(\bigcap_{s\in J}\ker\big(\fX^\lambda(T_s+v^{-1})\big)\Big)=1 .$$
Now we are going to use the fact that $\fX^\lambda$ is a
$W$-graph representation:
Using the $I$-sets associated with $\fX^\lambda$, see Definition \ref{wgraph}, 
we conclude that $\ker(\fX^\lambda(T_s+v^{-1}))=\langle e_i;s\in I_i\rangle_K$
for all $s\in S$, where $e_i\in K^n$ denotes the $i$-th ``unit'' vector.
This implies
$$ \bigcap_{s\in J}\ker\big(\fX^\lambda(T_s+v^{-1})\big)
=\langle e_i;J\subseteq I_i\rangle_K .$$
Hence we may let $u_1:=e_i$, where $1\leq i\leq n$ is the unique index
such that $J\subseteq I_i$.

Note that this conversely also yields a way to find all subsets 
of $S$ fulfilling the Benson--Curtis condition:
We run through all subsets $J\subseteq S$, and just check whether
there is precisely one index $1\leq i\leq n$ such that $J\subseteq I_i$.
\end{abs}

\absT{Finding a distinguished element}\label{finddist}
The above immediate approach strongly uses the fact that 
$\fX^\lambda$ is a $W$-graph representation. Thus, in
order to find a suitable 
seed vector $u'_1\in K^n$ for the standard basis algorithm with respect
to $(\fX^\lambda)'$ we specify a distinguished element $T^\lambda\in\cH_K$ 
such that $\dim_K(\ker(\fX^\lambda(T^\lambda)))=1$. Let
$$ T^\lambda:=\big(\sum_{s\in J}T_s\big)+v^{-1}\cdot|J| 
   \in\cH_A\subseteq\cH_K .$$
Hence we have 
$\bigcap_{s\in J}\ker(\fX^\lambda(T_s+v^{-1}))
\leq\ker(\fX^\lambda(T^\lambda))$,
and it remains to be shown that 
$\dim_K(\ker(\fX^\lambda(T^\lambda)))=1$: 

Assume to the contrary that $\dim_K(\ker\big(\fX^\lambda(T^\lambda)))\geq 2$. 
Then letting 
$$ \sigma_J:=\frac{1}{|J|}\cdot\sum_{s\in J}s\in\Q[\widetilde W] ,$$ 
specializing $v\mapsto 1$ shows that 
$\dim_{\Q}(\ker(\mathfrak Y^\lambda(1+\sigma_J)))\geq 2$ as well.
Since for any vector $u\in\ker(\mathfrak Y^\lambda(1+\sigma_J))$ we have 
$u\cdot\mathfrak Y^\lambda(\sigma_J^k)=(-1)^k\cdot u$, for all $k\in\N_0$, 
Lemma \ref{converge} proven below implies that 
$\langle u\rangle_\Q\leq K^n$ is $\Q[\widetilde W]$-invariant
and carries the sign representation. Thus we have
$\dim_\Q(\Hom_{\Q[\widetilde W]}(
\operatorname{sgn}_{\widetilde W},\mathfrak Y^\lambda))\geq 2$,
a contradiction.
\end{abs}

\begin{lem}\label{converge}
For $\epsilon\in\{0,1\}$ let 
$W_\epsilon:=\{w\in W;\operatorname{sgn}(w)=(-1)^\epsilon\}$.
Moreover, let 
$$ \sigma_S:=\frac{1}{|S|}\cdot\sum_{s\in S}s\in\Q[W] .$$
Then, with respect to the natural topology on $\Q[W]\cong\Q^{|W|}$, we have
$$ \lim_{k\rightarrow\infty}\sigma_S^{2k+\epsilon}
=\frac{1}{|W_\epsilon|}\cdot\sum_{w\in W_\epsilon}w\in\Q[W] .$$
\end{lem}

\begin{proof}
We consider the Markov chain with (finite) state space 
$W=W_0\stackrel{.}{\cup}W_1$, 
and transition matrix $M=\operatorname{reg}_W(\sigma_S)\in\Q^{|W|\times |W|}$, 
where $\operatorname{reg}_W\colon\Q[W]\rightarrow\Q^{|W|\times |W|}$ denotes
the regular matrix representation of $\Q[W]$. In other words,
the matrix entry $M_{w,w'}$, where $w,w'\in W$, is given as
$$ M_{w,w'}:=\left\{\begin{array}{rl}
\frac{1}{|S|}, & 
\quad\textrm{if}\quad w'=ws\quad\textrm{for some}\quad s\in S, \\
0, & \quad\textrm{otherwise}. \rule{0em}{1.2em} \\
\end{array} \right. $$
Now, since $\operatorname{sgn}(ws)=-\operatorname{sgn}(w)$ 
for all $w\in W$ and $s\in S$, we conclude that 
$M^2=\operatorname{reg}_W(\sigma_S^2)$ induces Markov chains on both 
$W_0$ and $W_1$. Moreover, since any element of $W$ can be written as
a word of length at most $l(w_0)$ in the generators $S$, we infer that
$M^{2l(w_0)}$ has positive entries in both the block submatrices belonging 
to $W_0$ and $W_1$, respectively. Hence the induced Markov chains are
both irreducible and aperiodic. They thus converge towards stationary
distributions, which since $M$ is doubly-stochastic are both equal to 
the respective uniform distributions.
Thus, in particular, the initial state 
$\sigma_S^\epsilon\in\langle W_\epsilon\rangle_\Q$ yields
$$ \lim_{k\rightarrow\infty}\sigma_S^{2k+\epsilon}
=\sigma_S^\epsilon\cdot\big(\lim_{k\rightarrow\infty}(M^2)^k\big) 
=\frac{1}{|W_\epsilon|}\cdot\sum_{w\in W_\epsilon}w .$$ 
\qed\end{proof}

\absT{Finding standard bases}\label{findstdbas}
The distinguished element $T^\lambda$ can now be used to find a 
primitive vector $u'_1\in\ker((\fX^\lambda)'(T^\lambda))$.
Next, having both seed vectors $u_1$ and $u'_1$ in place,
we aim at computing the associated standard bases $\mathfrak B$
with respect to $\fX^\lambda$, and $\mathfrak B'$ with respect
to $(\fX^\lambda)'$, for the $A$-algebra generated by $\{vT_s;s\in S\}$.
But since we do not have a standard basis algorithm available for
representations over the field $K$, we again use suitable specializations:

Given a place $0\neq b\in\Z$, let 
$\mathfrak Y^\lambda_b\colon\cH_\Q\rightarrow\Q^{n\times n}$ be the
representation of $\cH_\Q$ obtained by specializing $v\mapsto b$,
that is, considering $\cH_\Q$ as the $\Q$-algebra generated by 
$\{bT_s;s\in S\}$ we have
$$ \mathfrak Y^\lambda_b\colon bT_s\mapsto
\big(\fX^\lambda(vT_s)\big)(b)
:=\fX^\lambda(vT_s)|_{v\mapsto b}\in\Z^{n\times n} ;$$
thus in particular for $b=1$, identifying $\cH_\Q$ with $\Q[W]$, we recover 
$\mathfrak Y^\lambda_1=\mathfrak Y^\lambda$.

Now we compare a putative run of the standard basis algorithm, 
as described in Section \ref{stdbas}, with respect to the 
seed vector $u_1\in\Z[v]^n$ and the generators 
$\{\fX^\lambda(vT_s)\in\Z[v]^{n\times n};s\in S\}$,
with a run with respect to the specialized seed vector $u_1(b)\in\Z^n$ 
and the generators $\{\mathfrak Y_b^\lambda(bT_s)\in\Z^{n\times n};s\in S\}$.
These successively produce standard bases $\mathfrak B\subseteq\Z[v]^n$ 
and $\mathfrak C\subseteq\Z^n$, respectively. We show by induction on
the cardinality $0\leq m\leq n$ of the intermediate sets $\mathfrak B$,
that for all but finitely many $b$ the set $\mathfrak C$ is obtained 
by specializing $\mathfrak B$, and that the Schreier trees found
in both runs coincide:

Indeed, the key steps are to decide for some 
$w:=u\cdot\fX^\lambda(vT_s)\in\Z[v]^n$ whether or not 
$w\in\langle\mathfrak B\rangle_K$, and similarly for its specialization
$w(b):=u(b)\cdot\mathfrak Y_b^\lambda(bT_s)\in\Z^n$ whether or not 
$w(b)\in\langle\mathfrak C\rangle_\Q$.
Identifying $\mathfrak B$ and $\mathfrak C$ with matrices
$B\in\Z[v]^{m\times n}$ and $C\in\Z^{m\times n}$, respectively, 
we have $C=B(b)$.
Considering the matrix $B_w\in\Z[v]^{(m+1)\times n}$
obtained by concatenating $B$ and $w$,
we have $w\not\in\langle\mathfrak B\rangle_K$ if and only if
there is an $((m+1)\times (m+1))$-submatrix $B'$ of $B_w$ such that
$\det(B'_w)\neq 0$. 
Similarly, we have $w(b)\not\in\langle\mathfrak C\rangle_\Q$ if and only if
there is an $((m+1)\times (m+1))$-submatrix $C'$ of 
$C_{w(b)}=B_w(b)\in\Z^{(m+1)\times n}$ such that $\det(C')\neq 0$. 
Hence, whenever $w(b)\not\in\langle\mathfrak C\rangle_\Q$ we also have  
$w\not\in\langle\mathfrak B\rangle_K$, and conversely for all but finitely 
many $b$ from $w\not\in\langle\mathfrak B\rangle_K$ we may conclude that
$w(b)\not\in\langle\mathfrak C\rangle_\Q$.
(We have used a similar argument in Section \ref{polnullspace}.) 

Thus assuming that $0\neq b\in\Z$ is suitably chosen,
we may just run the standard basis algorithm for the 
seed vector $u_1(b)=u_1=e_i\in\Z^n$, the $i$-th ``unit'' vector,
and the generators $\mathfrak Y_b^\lambda(bT_s)\in\Z^{n\times n}$,
as described in Section \ref{stdbas}, yielding a Schreier tree $\mathfrak T$. 
Letting $w_1:=1\in W$, and $w_i:=w_j\cdot s\in W$, 
if $[j,s]$ is the $i$-th entry in $\mathfrak T$, for $2\leq i\leq n$,
we thus obtain reduced expressions of the elements
$w_i\in W$, and hence the number of steps needed to find the $i$-th
element of $\mathfrak C$ equals the length $l(w_i)\in\N_0$.
(In practice, it turns out that choosing either $b=1$ or $b=2$ is sufficient,
where actually almost always $b=1$ works.)

Applying the Schreier tree $\mathfrak T$ to $u_1$ and 
$\{\fX^\lambda(vT_s);s\in S\}$ this yields a standard basis
$\mathfrak B\subseteq\Z[v]^n$ of $K^n$.
Similarly, applying $\mathfrak T$ to $u'_1\in\Z[v]^n$ and 
$\{(\fX^\lambda)'(vT_s)\in\Z[v]^{n\times n};s\in S\}$
we get a standard basis $\mathfrak B'\subseteq\Z[v]^n$ of $K^n$. 
But note that this does \emph{not} ensure that the $A$-lattices 
$\langle\mathfrak B\rangle_A$ and $\langle\mathfrak B'\rangle_A$ 
are invariant under the $A$-algebras generated by 
$\{\fX^\lambda(vT_s);s\in S\}$ and $\{(\fX^\lambda)'(vT_s);s\in S\}$, 
respectively. (In practice they are not, typically.)
\end{abs}

\section{Finding Gram matrices for $W$-graph representations}\label{bilcompII}

We keep the setting of Section \ref{bilcomp}; in particular 
$\fX^\lambda$ still is a $W$-graph representation.
Having found standard bases $\mathfrak B$ and $\mathfrak B'$
for $\fX^\lambda$ and $(\fX^\lambda)'$, respectively,
we proceed by writing them as matrices $B\in\Z[v]^{n\times n}$ and
$B'\in\Z[v]^{n\times n}$, respectively, where by construction both 
$B$ and $B'$ are primitive. In order to complete the final task of
computing the product $B^{-1}\cdot B'\in\Z[v]^{n\times n}$ efficiently,
we need a few preparations.

\absT{Palindromicity}
Let $\ast\colon K\rightarrow K$ be the involutory field automorphism 
given by $\ast\colon v\mapsto v^{-1}$. Hence $A$ is $\ast$-invariant,
and by entry-wise application we get involutory module automorphisms on 
$K^n$ and $A^n$, and algebra automorphisms on $K^{n\times n}$ and
$A^{n\times n}$, all of which will also be denoted by $\ast$.

A polynomial $0\neq f\in\Z[v]$ is called \emph{($k$-)palindromic}, 
for some $k\in\N_0$, if $v^k\cdot f^\ast=f\in A$, and $f$ is 
called \emph{($k$-)skew-palindromic} if $v^k\cdot f^\ast=-f\in A$.
In these cases, letting $\delta(f)\in\N_0$ be the maximum power
of $v$ dividing $f$ in $\Z[v]$, we have $k=\delta(f)+\deg(f)$.
Hence $f$ is palindromic or skew-palindromic if and only if
$f\in\Z[v]$ and $f^\ast\in\Z[v^{-1}]$ are associated in $A$.
Moreover, if $f$ is $k$-skew-palindromic, then specializing $v\mapsto 1$ 
we get $f(1)=-f(1)$, implying that $v-1$ divides $f$ in $\Z[v]$;
similarly, if $f$ is $k$-palindromic, then specializing $v\mapsto -1$ we get 
$(-1)^k\cdot f(-1)=f(-1)$, implying that $k$ is even,
or $v+1$ divides $f$ in $\Z[v]$.
\end{abs}

\begin{prop}\label{palindromic}
(a)
Let $P\in\Z[v]^{n\times n}$ be a primitive Gram matrix for $\fX^\lambda$.
Then we have $v^m\cdot P^\ast=P$, where $m=m_P\in\N$ is even and coincides 
with the maximum of the degrees of the non-zero entries of $P$.

(b)
For the primitive seed vector $u'_1\in\Z[v]^n$ we have 
$v^m\cdot (u'_1)^\ast=u'_1$, where $m=m_{u'_1}\in\N_0$ is even and coincides
with the maximum of the degrees of the non-zero entries of $u'_1$.
(Trivially, the analogous statement holds for $u_1\in\Z[v]^n$ 
with $m_{u_1}=0$.)
\end{prop}

\begin{proof}
Letting $E_n\in A^{n\times n}$ be the identity matrix, by 
Definition \ref{wgraph} for $s\in S$ we have
$$ \fX^\lambda(T_s)^\ast=\fX^\lambda(T_s)-(v-v^{-1})\cdot E_n
=\fX^\lambda\big(T_s-(v-v^{-1})\big) .$$
In particular, this yields
$$ \fX^\lambda(T_s+v^{-1})^\ast 
=\fX^\lambda(T_s)^\ast+v\cdot E_n
=\fX^\lambda\big(T_s-(v-v^{-1})\big)+v\cdot E_n
=\fX^\lambda(T_s+v^{-1}) .$$

(a)
We consider the matrix $P^\ast\in\Z[v^{-1}]^{n\times n}$:
For all $s\in S$ we have
$$ \fX^\lambda(T_s)\cdot P^\ast 
= \Big(\fX^\lambda\big(T_s-(v-v^{-1})\big)\cdot P\Big)^\ast 
= \Big(P\cdot
  \fX^\lambda\big(T_s-(v-v^{-1})\big)^{\operatorname{tr}}\Big)^\ast $$
$$ = \Big(P\cdot\fX^\lambda(T_s)^{\ast\operatorname{tr}}\Big)^\ast 
= \Big(P\cdot\fX^\lambda(T_s)^{\operatorname{tr}\ast}\Big)^\ast 
= P^\ast\cdot\fX^\lambda(T_s)^{\operatorname{tr}} .$$ 
Now $m=m_P\in\N$ as above is minimal such that
$v^m P^\ast\in\Z[v]^{n\times n}$, hence we infer that 
$v^m P^\ast$ is a primitive Gram matrix for $\fX^\lambda$
as well, and thus we have $v^m P^\ast=P$ or $v^m P^\ast=-P$. 
Assume the latter case holds, then all non-zero entries of $P$ are 
$m$-skew-palindromic, implying that $v-1$ divides $\gcd(P)$,
contradicting the primitivity of $P$. Hence we have $v^m P^\ast=P$,
that is all non-zero entries of $P$ are $m$-palindromic.
Assume that $m$ is odd, then we infer that $v+1$ divides $\gcd(P)$,
again contradicting the primitivity of $P$.
Hence $m$ is even.

(b) We consider the vector $(u'_1)^\ast\in\Z[v^{-1}]^n$: We have 
$$ (u'_1)^\ast\cdot(\fX^\lambda)'(T^\lambda)
=\Big(u'_1\cdot(\fX^\lambda)'(T^\lambda)^\ast\Big)^\ast 
=\Big(u'_1\cdot\big(
  \sum_{s\in J}\fX^\lambda(T_s+v^{-1})\big)^{\operatorname{tr}\ast}\Big)^\ast 
$$
$$ = \Big(u'_1\cdot\big(
     \sum_{s\in J}\fX^\lambda(T_s+v^{-1})\big)^{\operatorname{tr}}\Big)^\ast 
= \Big(u'_1\cdot (\fX^\lambda)'(T^\lambda)\Big)^\ast = 0 .$$
Now $m=m_{u'_1}\in\N_0$ as above is minimal such that 
$v^m\cdot(u'_1)^\ast\in\Z[v]^n$, 
hence we infer that $v^m\cdot(u'_1)^\ast$ is primitive. Thus from
$\dim_K(\ker((\fX^\lambda)'(T^\lambda)))=1$ we conclude that 
$v^m\cdot(u'_1)^\ast=u'_1$ or $v^m\cdot(u'_1)^\ast=-u'_1$.
Now we argue as above.
\qed\end{proof}

\absT{Properties of the standard bases}\label{stdbasprop}
We have a closer look at the standard bases $\mathfrak B$ and 
$\mathfrak B'$, and the associated matrices $B$ and $B'$, where 
we assume $\mathfrak B$ to be chosen according to Section \ref{findstdbas}.
The facts collected are largely due to experimental observation,
and will be helpful in the final computational steps in Section \ref{final}.
Still, these properties seem to be stronger than expected from 
general principles, and it should be worth-while to try and prove
the particular observations specified below. 
(In particular, we have checked the standard bases associated 
with \emph{all} subsets $J\subseteq S$ fulfilling the Benson--Curtis
condition, see Section \ref{seedvec}, for the types $E_6$, $E_7$ and $E_8$.)

Recall that for all $s\in S$ we have
$$ (vT_s)^{-1} = v^{-1}\cdot\big(T_s-(v-v^{-1})\big)
=v^{-2}\cdot\big(vT_s-(v^2-1)\big) ,$$ 
hence by the proof of Proposition \ref{palindromic} we get
$$ \fX^\lambda(vT_s)^\ast
=v^{-1}\cdot\fX^\lambda\big(T_s-(v-v^{-1})\big)
=v^{-2}\cdot\fX^\lambda\big(vT_s-(v^2-1)\big)
=\fX^\lambda\big((vT_s)^{-1}\big) .$$
\end{abs}

\subT{The elements of $\mathfrak B$}
For any $u_i\in\mathfrak B$, where $2\leq i\leq n$, we have
$u_i=u_j\cdot\fX^\lambda(vT_s)$, for some $1\leq j<i$ and $s\in S$.
This yields 
$$ v^2\cdot u_j = v^2\cdot u_i\cdot\fX^\lambda\big((vT_s)^{-1}\big)
= u_i\cdot\fX^\lambda\big(vT_s-(v^2-1)\big) .$$
We conclude that $\gcd(u_i)\in\Z[v]$ and $\gcd(u_j)\in\Z[v]$
are associated in $A$. Hence by recursion, since $u_1$ is primitive,
we infer that $\gcd(u_i)=v^{d_i}\in\Z[v]$ for some $d_i\in\N_0$.

Moreover, we have $d_j\leq d_i\leq d_j+2$. Since $d_1=0=l(w_1)$, 
this implies $d_i\leq 2l(w_i)$ for all $1\leq i\leq n$, where $w_i\in W$ 
is as in Section \ref{findstdbas}. (Experiments show that 
all three cases $d_i\in\{d_j,d_j+1,d_j+2\}$ actually occur.)
But the growth behavior of the $d_i$ seems to be 
more restricted than given by these bounds:
Considering the case $l(w_i)=1$, 
we have $w_i=s$ for some $s\in S$ such that the ``unit'' vector
$u_1$ is not an eigenvector of $T_s$, hence using the shape of     
$\fX^\lambda(vT_s)$ we conclude that $d_i=1=l(w_i)$.

Now, experimentally, we have made the following 

\begin{obs}\label{obs1}
We have $d_i\leq l(w_i)+1$, for all $1\leq i\leq n$.
\end{obs}

(Actually, almost always we have got $d_i\leq l(w_i)$, for all
$1\leq i\leq n$, where often we have even seen equality throughout;
the only cases found where actually $d_i=l(w_i)+1$, for some $i$, are 
for type $E_8$, the representation labeled by $3200_x$, and 
two out of the twelve Benson--Curtis subsets of generators.)
\end{sub}

\subT{The matrix $B$}
Letting $1\leq j<i\leq n$ and $s\in S$ be as above, we get 
$$ v^2\cdot u_i^\ast 
= v^2\cdot u_j^\ast\cdot\fX^\lambda(vT_s)^\ast 
= u_j^\ast\cdot\fX^\lambda\big(vT_s-(v^2-1)\big) .$$
Since the standard basis algorithm is a breadth-first search, 
from $u_1^\ast=u_1$ we conclude that there is lower unitriangular matrix
$U\in K^{n\times n}$ and a diagonal matrix
$D=\operatorname{diag}[v^{2l(w_1)},\ldots,v^{2l(w_n)}]\in\Z[v]^{n\times n}$,
such that
$$ D\cdot B^\ast = U\cdot B .$$
(Note that if the $A$-lattice $\langle\mathfrak B\rangle_A$
was invariant under the $A$-algebra generated by 
$\{\fX^\lambda(vT_s);s\in S\}$, then we even had $U\in A^{n\times n}$.)

In particular, letting $l:=\sum_{i=1}^n l(w_i)\in\N_0$, we infer that
$$ \det(B)=v^{2l}\cdot\det(B^\ast) ,$$
hence $\det(B)\in\Z[v]$ is palindromic. 
Letting $\exp(B)\in\Z[v]$ denote 
the exponent of $B$ in the sense of Section \ref{polexp}, 
it follows from Proposition \ref{polexpprop} that the 
non-constant irreducible polynomials dividing $\det(B)$ are
precisely those dividing $\exp(B)$. Now, experimentally, 
we have made the following

\begin{obs}\label{obs2}
Any irreducible divisor of $\exp(B)$ in $\Z[v]$ is 
monic and palindromic.
\end{obs}

(Actually, in general the entries of the matrix $B$ are 
neither palindromic nor skew-\-palindromic; moreover,
quite often $\exp(B)$ is a product of cyclotomic polynomials,
but this does not always happen.)

In particular, if $\widehat u_k^{\operatorname{tr}}\in\Z[v]^{1\times n}$
denotes the $k$-th column of $B$, for $1\leq k\leq n$,
then $\gcd(\widehat u_k)\in\Z[v]$ divides $\det(B)$, hence 
$\gcd(\widehat u_k)$ is palindromic as well.
(Actually, contrary to $\gcd(u_k)=v^{d_k}$, 
in general the $\gcd(\widehat u_k)$ are not just powers of $v$.)
\end{sub}

\subT{The elements of $\mathfrak B'$}
The recursion used in the standard basis algorithm 
only depends on the Schreier tree $\mathfrak T$,
but is independent of the representation considered. Hence for 
$u'_i\in\mathfrak B'$, where $1\leq i\leq n$, and $u'_1$ is 
primitive, we get $\gcd(u'_i)=v^{d'_i}\in\Z[v]$ for some $d'_i\in\N_0$.
Moreover, if $1\leq j<i\leq n$ and $s\in S$ are as above, we get
$d'_j\leq d'_i\leq d'_j+2$ and $d'_i\leq 2l(w_i)$.
Actually, the $d'_i$ seem to be closely related to the $d_i$ from above, 
inasmuch experimentally we have made the following 

\begin{obs}\label{obs3}
We have $d'_i=d_i$, for all $1\leq i\leq n$. 
\end{obs}
\end{sub}

\subT{The matrix $B'$}
Again by the fact that the recursion used in the standard basis algorithm 
only depends on $\mathfrak T$, and using $v^m\cdot (u'_1)^\ast=u'_1$, where
$m=m_{u'_1}\in\N_0$ is as in Proposition \ref{palindromic}, we get
$$ v^m\cdot D\cdot (B')^\ast = U\cdot B' ,$$
for the same matrices $U$ and $D$. In particular, it follows that 
$\det(B')$ is palindromic.
(In general neither $\det(B')$ and $\det(B)$, nor $\exp(B')$ 
and $\exp(B)$ are associated in $A$,
so that $\langle\mathfrak B\rangle_A$ and $\langle\mathfrak B'\rangle_A$
are inequivalent $A$-sublattices of $A^n$, which typically
are not included in each other.)
Again, experimentally we have made the following

\begin{obs}\label{obs4}
Any irreducible divisor of $\exp(B')$ in $\Z[v]$ is 
monic and palindromic.
\end{obs}

In particular, similarly, if 
$\widehat u_k^{\prime\operatorname{tr}}\in\Z[v]^{1\times n}$
denotes the $k$-th column of $B'$, for $1\leq k\leq n$, 
then $\gcd(\widehat u_k^{\prime})\in\Z[v]$ is palindromic.
\end{sub}

\subT{The product $B^{-1}\cdot B'$}
In combination the above yields
$$ v^m\cdot (B^{-1}\cdot B')^\ast 
= v^m\cdot (B^\ast)^{-1}\cdot (B')^\ast 
= (D^{-1}\cdot U\cdot B)^{-1}\cdot (D^{-1}\cdot U\cdot B')
= B^{-1}\cdot B' .$$
Hence the non-zero entries of $B^{-1}\cdot B'$ are palindromic. 

Letting $0\neq b\in\Z$ and
$\widehat B\in\Z[v]^{n\times n}$ primitive such that 
$B^{-1}=\frac{1}{b\cdot\exp(B)}\cdot\widehat B$,
we get 
$$ b\cdot\exp(B)\cdot B^{-1}\cdot B'
=\widehat B\cdot B'=c\cdot P ,$$
where $P\in\Z[v]^{n\times n}$ is a primitive Gram matrix, 
and $0\neq c\in\Z[v]$.
In particular, since by Observation \ref{obs2} 
the exponent $\exp(B)$ is palindromic,
we conclude that the non-zero entries of $\widehat B\cdot B'$ 
are palindromic as well.
 
Moreover, letting $\widetilde m=m_{\exp(B)}\in\N_0$
such that $v^{\widetilde m}\cdot\exp(B)^\ast=\exp(B)$, we get
$$ v^{m+\widetilde m}\cdot(b\cdot\exp(B)\cdot B^{-1}\cdot B')^\ast
= b\cdot\exp(B)\cdot B^{-1}\cdot B'\in\Z[v]^{n\times n} .$$
Hence from $v^{m_P}\cdot P^\ast=P$,
where $m_P\in\N_0$ is as in Proposition \ref{palindromic}, we get 
$$ m_P \leq m+\widetilde m=m_{u'_1}+m_{\exp(B)} ,$$ 
providing an upper bound on the degrees of the non-zero entries of $P$.   
\end{sub}

\absT{The final product}\label{final}
We are now prepared to do the last computational steps. To do so,
we could quite straightforwardly compute first the inverse $B^{-1}$,
that is essentially $\widehat B$, and then the product $\widehat B\cdot B'$.
But it will substantially add to the efficiency if we keep the degrees
of the non-zero entries of the matrices involved as small as possible.
Now we have already observed above that the rows of $B$ and $B'$ are
far from being primitive, and it turns out in practice that this 
also holds for their columns. We take advantage of this as follows:

Keeping the notation of Section \ref{stdbasprop}, let
$R:=\operatorname{diag}[v^{d_1},\ldots,v^{d_n}]\in\Z[v]^{n\times n}$.
Then the rows of $R^{-1}\cdot B\in\Z[v]^{n\times n}$ are primitive.
As for its columns, 
letting $\widetilde u_k^{\operatorname{tr}}\in\Z[v]^{1\times n}$
denote the $k$-th column of $R^{-1}\cdot B$, for $1\leq k\leq n$, let 
$$ C:=\operatorname{diag}[\gcd(\widetilde u_1),\ldots,\gcd(\widetilde u_n)]
    \in\Z[v]^{n\times n} .$$
Since by Observation \ref{obs2} the polynomial $\gcd(\widehat u_k)$ 
is palindromic, using the particular form of $R$, we conclude that the 
$\gcd(\widetilde u_k)$ are palindromic as well. 
We let $0\neq\widehat c\in\Z[v]$ and $\widehat C\in\Z[v]^{n\times n}$ 
be primitive such that $C^{-1}=\frac{1}{\widehat c}\cdot\widehat C$.
The latter are of course straightforwardly computed, where both 
$\widehat c$ and the diagonal entries of $\widehat C$ are palindromic.

Then we get $\widetilde B\in\Z[v]^{n\times n}$ such that 
$B=R\cdot\widetilde B\cdot C$, where now all the rows and all the columns
of $\widetilde B$ are primitive.
We use the algorithm in Section \ref{polinverse} to compute 
$0\neq\widehat b\in\Z[v]$ and $\widehat B\in\Z[v]^{n\times n}$ primitive 
such that $\widetilde B^{-1}=\frac{1}{\widehat b}\cdot\widehat B$,
Since by Observation \ref{obs2} the exponent $\exp(B)$ is palindromic,
using the particular form of $R$ and $C$, 
we conclude that $\widehat b$ is palindromic as well.
Thus altogether we have
$$ B^{-1}=\frac{1}{\widehat b\cdot \widehat c}\cdot 
   \widehat C\cdot\widehat B\cdot R^{-1} .$$

Similarly, let 
$R':=\operatorname{diag}[v^{d'_1},\ldots,v^{d'_n}]\in\Z[v]^{n\times n}$ and 
$$ C':=\operatorname{diag}[\gcd(\widetilde u'_1),\ldots,\gcd(\widetilde u'_n)]
       \in\Z[v]^{n\times n} ,$$
where $\widetilde u_k^{\prime\operatorname{tr}}\in\Z[v]^{1\times n}$
denotes the $k$-th column of $(R')^{-1}\cdot B'$, for $1\leq k\leq n$.
As above, using Observation \ref{obs4} implying the palindromicity of
$\gcd(\widehat u'_k)$, we conclude that the diagonal entries
of $C'$ are palindromic as well, and thus those of $(C')^{-1}$ are too.
Then we get $\widetilde B'\in\Z[v]^{n\times n}$
such that $B'=R'\cdot\widetilde B'\cdot C'$, where now all the rows and all 
the columns of $\widetilde B$ are primitive.

In combination this yields
$$ Q:= \widehat b\cdot \widehat c\cdot B^{-1}\cdot B'
=\widehat C\cdot\widehat B\cdot R^{-1}\cdot R'\cdot\widetilde B'\cdot C' .$$
By the above considerations we conclude that the non-zero entries of $Q$
are palindromic, which entails that those of 
$\widehat B\cdot R^{-1}\cdot R'\cdot\widetilde B'$ are as well. 
Now by Observation \ref{obs3} we have $R'=R$, hence this simplifies to 
$$ Q = \widehat C\cdot(\widehat B\cdot\widetilde B')\cdot C'
   \in\Z[v]^{n\times n} ,$$
where the non-zero entries of 
$\widehat B\cdot\widetilde B'\in\Z[v]^{n\times n}$ are palindromic.
\end{abs}

\subT{In practice}
To find $Q$, finally, we apply the matrix multiplication algorithm
in Section \ref{polproduct} to compute the product 
$\widehat B\cdot\widetilde B'$. 
As was already mentioned, in order to 
apply it efficiently we need good break conditions to discard 
erroneous guesses quickly: Apart from requiring that rational number
recovery, see Section \ref{ratrecog}, returns only integral coefficients
but not rational ones, it turns out that checking for palindromicity
is highly effective in this respect. 

Having found a good candidate for 
$\widehat B\cdot\widetilde B'\in\Z[v]^{n\times n}$,
multiplying with the diagonal matrices $\widehat C\in\Z[v]^{n\times n}$ 
and $C'\in\Z[v]^{n\times n}$ is straightforward.
Note that, since the result is expected to be a symmetric matrix,
it is sufficient to compute only the lower triangular half
of the product.
Thus we get a candidate for a primitive Gram matrix $P$ from 
$Q=\gcd(Q)\cdot P\in\Z[v]^{n\times n}$.
(In many cases $Q$ already is primitive, but this does not 
happen always, in which cases $\gcd(Q)$ typically has a 
smallish degree.)

As independent verification we of course just explicitly check 
whether the candidate $P$ fulfills the condition
$$ \fX^\lambda(vT_s)\cdot P=P\cdot\fX^\lambda(vT_s)^{\operatorname{tr}}
\in\Z[v]^{n\times n}
\quad\textrm{for all}\quad 
s\in S .$$
\end{sub}

\section{Timings}\label{sec:timings}

We conclude by providing running times and workspace requirements for our
computations in types $E_7$ and $E_8$, and by presenting an explicit 
example for type $E_6$.  

\absT{Timings} \label{timings}
In Table \ref{timingstbl}, we give the running time (on a single
processor running at a clock speed of $3.5\textrm{GHz}$) and 
\GAP{} workspace requirements needed to compute
primitive Gram matrices for types $E_7$ and $E_8$,
and the irreducible $W$-graph representations 
of $\cH_K$ given in \cite{How}, \cite{HowYin}.
The figures for $E_7$ should be compared with those 
given in Section \ref{wgr1a} for the approach used there.
Recalling that in \cite[Remark 4.10]{gemu}
degree $2500$ was the limit of feasibility, in Table \ref{timingstbl2}
we present the resources now needed for the individual 
representations of degree at least $2500$,
where for comparison we repeat the first three columns of the 
relevant part of Table \ref{Mmaxd}.

Finally, in Table \ref{timingstbl3} we give some details about the
various steps in the computation for the unique representation of
largest degree, which is labeled by $7168_w$.
In the two last columns we indicate the actual size of
the object under consideration in the \GAP{} workspace, and 
the disc space needed to store it (as an uncompressed text file), 
respectively; the difference is accounted for by the
space consumption of the data structure we are using within \GAP{},
where matrices with polynomial entries are kept as 
lists of lists of (short) lists of (small long) integers.
In particular, in the workspace needed to compute the product,
next to the matrices $\widehat B$ and $\widetilde B'$ and 
(the lower triangular half of) the product $\widehat B\cdot\widetilde B'$, 
we also keep various specializations of the right hand factor
$\widetilde B'$, which have a cumulative size of $7.1\textrm{GB}$.
Hence to compute a primitive Gram matrix for the representation
labeled by $7168_w$ we need a running time of
$1183\textrm{min}\sim 20\textrm{h}$ and a workspace of size 
$31.5\textrm{GB}$.  
\end{abs}

\begin{table}\caption{Time and space consumption}\label{timingstbl}
$$ \begin{array}{|ccc|rr|}
\hline
\rule{2em}{0em} & \rule{1.5em}{0em}\textrm{degree}\rule{1.5em}{0em} 
                & \rule{0.5em}{0em}\textrm{no.~repr.}\rule{0.5em}{0em} 
                & \rule{2em}{0em}\textrm{time} 
                & \rule{1em}{0em}\textrm{workspace}\rule{0em}{1em} \\
\hline
\hline
E_7 & \textrm{all} & 60 & 4\textrm{min} & 0.2\textrm{GB} \rule{0em}{1em} \\
\hline
\hline
E_8 & \leq 1000 & 50 & 30\textrm{min} & 0.7\textrm{GB} \rule{0em}{1em} \\
    & 1000\textrm{---}2000 & 20 & 137\textrm{min} 
    & 2.2\textrm{GB} \rule{0em}{1em} \\
    & 2000\textrm{---}2500 & 10 & 329\textrm{min} 
    & 4.3\textrm{GB} \rule{0em}{1em} \\
    & 2500\textrm{---}3000 & 5 & 350\textrm{min} 
    & 5.9\textrm{GB} \rule{0em}{1em} \\
    & 3000\textrm{---}4000 & 7 & 874\textrm{min} 
    & 11.6\textrm{GB} \rule{0em}{1em} \\
    & 4000\textrm{---}5000 & 13 & 3175\textrm{min} 
    & 16.3\textrm{GB} \rule{0em}{1em} \\
    & 5000\textrm{---}7000 & 6 & 2784\textrm{min} 
    & 23.2\textrm{GB} \rule{0em}{1em} \\
    & \geq 7000 & 1 & 1183\textrm{min} & 31.5\textrm{GB} \rule{0em}{1em} \\
\hline
\end{array} $$
\end{table}

\begin{table}\caption{Time and space consumption for degree $\geq 2500$}
\label{timingstbl2}
$$ \begin{array}{|ccc|rr|}
\hline
\rule{1em}{0em}E_8\rule{1em}{0em} 
    & \rule{0.5em}{0em}m_P\rule{0.5em}{0em} 
    & \rule{1em}{0em}\textrm{abs.val.}\rule{1em}{0em} 
    & \rule{2em}{0em}\textrm{time} 
    & \rule{1em}{0em}\textrm{workspace}\rule{0em}{1em} \\
\hline
\hline
2688_y & 24 & 169180 & 39\textrm{min} & 3.9\textrm{GB}\rule{0em}{1em}\\ 
2800_z & 20 & 38038 & 61\textrm{min} & 3.7\textrm{GB}\rule{0em}{1em}\\ 
2800_z' & 30 & 882222 & 116\textrm{min} & 5.9\textrm{GB}\rule{0em}{1em}\\ 
2835_x & 24 & 1344484 & 52\textrm{min} & 3.1\textrm{GB}\rule{0em}{1em}\\ 
2835_x' & 32 & 5391418 & 82\textrm{min} & 5.3\textrm{GB}\rule{0em}{1em}\\
\hline
\hline
3150_y & 26 & 6166994 & 72\textrm{min} & 5.8\textrm{GB}\rule{0em}{1em}\\ 
3200_x & 24 & 266284 & 79\textrm{min} & 4.9\textrm{GB}\rule{0em}{1em}\\ 
3200_x' & 30 & 587345 & 104\textrm{min} & 6.1\textrm{GB}\rule{0em}{1em}\\ 
3240_z & 16 & 25586 & 60\textrm{min} & 4.0\textrm{GB}\rule{0em}{1em}\\ 
3240_z' & 48 & 33653538&326\textrm{min}&11.6\textrm{GB}\rule{0em}{1em}\\ 
3360_z & 20 & 29722 & 74\textrm{min} & 5.1\textrm{GB}\rule{0em}{1em}\\ 
3360_z' & 32 & 775084 & 159\textrm{min} & 8.1\textrm{GB}\rule{0em}{1em}\\
\hline
\hline
4096_x & 22 & 531634 & 156\textrm{min} & 8.0\textrm{GB}\rule{0em}{1em}\\ 
4096_x' & 44 & 234956568&392\textrm{min}&16.0\textrm{GB}\rule{0em}{1em}\\ 
4096_z & 22 & 531634 & 143\textrm{min} & 8.1\textrm{GB}\rule{0em}{1em}\\ 
4096_z' & 44 &234956568&428\textrm{min}&16.1\textrm{GB}\rule{0em}{1em}\\ 
4200_y & 28 & 58249760 & 171\textrm{min}&10.1\textrm{GB}\rule{0em}{1em}\\ 
4200_x & 24 & 5413484 & 171\textrm{min} & 9.8\textrm{GB}\rule{0em}{1em}\\ 
4200_x' & 36 & 129331224&277\textrm{min}&13.3\textrm{GB}\rule{0em}{1em}\\ 
4200_z & 26 & 728053 &183\textrm{min} & 10.4\textrm{GB}\rule{0em}{1em}\\ 
4200_z' & 28 & 1298612 &199\textrm{min}&10.3\textrm{GB}\rule{0em}{1em}\\ 
4480_y & 32 &85556320920&239\textrm{min}&13.9\textrm{GB}\rule{0em}{1em}\\ 
4536_y & 28 & 3887856 & 180\textrm{min} &11.7\textrm{GB}\rule{0em}{1em}\\ 
4536_z & 24 & 2728756 & 217\textrm{min}&11.4\textrm{GB}\rule{0em}{1em}\\ 
4536_z' & 38 & 50779421&419\textrm{min}&16.3\textrm{GB}\rule{0em}{1em}\\ 
\hline
\hline
5600_w & 26 & 372230 & 331\textrm{min} & 16.6\textrm{GB}\rule{0em}{1em}\\ 
5600_z & 26 & 3115126 & 335\textrm{min}&15.4\textrm{GB}\rule{0em}{1em}\\ 
5600_z' & 30 & 3848044 &473\textrm{min}&17.5\textrm{GB}\rule{0em}{1em}\\ 
5670_y & 30 & 10762741 & 351\textrm{min}&21.7\textrm{GB}\rule{0em}{1em}\\ 
6075_x & 26 & 894864 & 542\textrm{min} & 19.5\textrm{GB}\rule{0em}{1em}\\ 
6075_x' & 34 & 10488013 &752\textrm{min}&23.2\textrm{GB}\rule{0em}{1em}\\ 
\hline
\hline
7168_w & 32&1190470476&1183\textrm{min}&31.5\textrm{GB}\rule{0em}{1em}\\ 
\hline
\end{array} $$
\end{table}

\begin{table}\caption{Time and space consumption for $7168_w$}
\label{timingstbl3}
$$ \begin{array}{|c|rr|rr|}
\hline
\rule{0.5em}{0em} 7168_w \rule{0.5em}{0em} 
 & \rule{2em}{0em} \textrm{time} 
 & \rule{1em}{0em} \textrm{workspace} & \rule{2em}{0em} \textrm{space} 
 & \rule{2em}{0em} \textrm{disc} \rule{0em}{1em} \\
\hline
\hline
\mathfrak T & 9\textrm{min} & 0.6\textrm{GB} & & \rule{0em}{1em} \\
u'_1 & 5\textrm{min} & 1.3\textrm{GB} \rule{0em}{1em} & & \\
\widehat B & 925\textrm{min} & 7.6\textrm{GB} & 1.7\textrm{GB} 
           & 0.3\textrm{GB} \rule{0em}{1em} \\
\widetilde B' & 29\textrm{min} & 17.5\textrm{GB} & 12.6\textrm{GB}
              & 4.7\textrm{GB} \rule{0em}{1em} \\
\widehat B\cdot\widetilde B' & 207\textrm{min} & 31.5\textrm{GB} 
         & 5.8\textrm{GB} & 2.4\textrm{GB} \rule{0em}{1em} \\
P & 8\textrm{min} & 7.9\textrm{GB} & 5.8\textrm{GB} 
  & 2.5\textrm{GB} \rule{0em}{1em} \\
\hline
\end{array} $$
\end{table}

\absT{An explicit example}\label{gramex}
We conclude by revisiting the (tiny) example already presented 
in \cite[Example 4.9]{gemu} (which of course in practice runs
in a fraction of a second):
Let $W$ be of type $E_6$ with Dynkin diagram

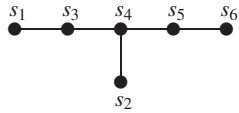
\begin{center}\small
\begin{picture}(200,30)
\put(60, 25){\circle*{5}}
\put(58, 30){$s_1$}
\put(60, 25){\line(1,0){20}}
\put(80, 25){\circle*{5}}
\put(78, 30){$s_3$}
\put(80, 25){\line(1,0){20}}
\put(100, 25){\circle*{5}}
\put(98, 30){$s_4$}
\put(100, 25){\line(0,-1){20}}
\put(100, 05){\circle*{5}}
\put(98, -4){$s_2$}
\put(100, 25){\line(1,0){20}}
\put(120, 25){\circle*{5}}
\put(118, 30){$s_5$}
\put(120, 25){\line(1,0){20}}
\put(140, 25){\circle*{5}}
\put(138, 30){$s_6$}
\end{picture}
\end{center}

We consider the irreducible $W$-graph representation of $\cH_K$,
see \cite{Naruse0}, labeled by the representation $10_s$ of $\Q[W]$,
which is the unique one of degree $10$, see Table \ref{Mmaxd0}. 
The $W$-graph in question is depicted in \cite[Example 4.9]{gemu}, hence we 
do not repeat it here. But to illustrate the shape, and in particular
the sparseness of the representing matrices for the generators 
$vT_{s_1},\ldots,vT_{s_6}$ we present a few of them:

{\small
$$ vT_1\mapsto\begin{bmatrix}
v^2&.&.&v&.&.&.&.&v&. \\
.&v^2&.&v&.&.&.&.&.&v \\
.&.&-1&.&.&.&.&.&.&. \\
.&.&.&-1&.&.&.&.&.&. \\
.&.&.&.&-1&.&.&.&.&. \\
.&.&v&.&.&v^2&.&.&.&. \\
.&.&.&.&.&.&v^2&.&v&v \\
.&.&.&.&v&.&.&v^2&.&. \\
.&.&.&.&.&.&.&.&-1&. \\
.&.&.&.&.&.&.&.&.&-1 \\
\end{bmatrix}
\quad vT_6\mapsto
\begin{bmatrix}
v^2&.&.&.&.&.&.&v&v&. \\
.&v^2&.&.&.&.&.&v&.&v \\
.&.&v^2&.&.&.&.&.&v&v \\
.&.&.&v^2&v&.&.&.&.&. \\
.&.&.&.&-1&.&.&.&.&. \\
.&.&.&.&.&v^2&v&.&.&. \\
.&.&.&.&.&.&-1&.&.&. \\
.&.&.&.&.&.&.&-1&.&. \\
.&.&.&.&.&.&.&.&-1&. \\
.&.&.&.&.&.&.&.&.&-1 \\
\end{bmatrix} $$ }


As it turns out, there are $22$ possible choices of a distinguished subset 
$J\subseteq S$. We choose $J:=\{s_1,s_2,s_3,s_5,s_6\}$, 
in accordance with \cite[Table C.4]{gepf}. Then associated primitive
seed vectors $u_1$ and $u'_1$ are as given below,
in the first row of the matrices $B$ and $\widetilde B'$, respectively.
Running the standard basis algorithm on the specialization of the 
above $W$-graph representation with respect to $v\mapsto 1$ yields
the following Schreier tree $\mathfrak T$, which we depict as an
oriented graph, whose vertices $1,\ldots,10$ correspond to the
vectors in the (ordered) standard bases, and where an arrow from vertex $j$
to vertex $i$ with label $s_k$ says that $[j,s_k]$ is the $i$-th 
entry of $\mathfrak T$:

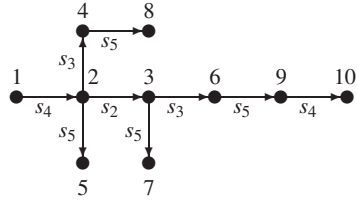
\begin{center}\small
\begin{picture}(125,60)
\put(-2,30){$1$}
\put(0,25){\circle*{5}}
\put(2,25){\vector(1,0){21}}
\put(7,19){$s_4$}
\put(27,30){$2$}
\put(25,25){\circle*{5}}
\put(27,25){\vector(1,0){21}}
\put(32,19){$s_2$}
\put(23,-12){$5$}
\put(25,0){\circle*{5}}
\put(25,23){\vector(0,-1){21}}
\put(16,10){$s_5$}
\put(23,55){$4$}
\put(25,50){\circle*{5}}
\put(25,27){\vector(0,1){21}}
\put(16,37){$s_3$}
\put(48,55){$8$}
\put(50,50){\circle*{5}}
\put(27,50){\vector(1,0){21}}
\put(32,44){$s_5$}
\put(48,30){$3$}
\put(50,25){\circle*{5}}
\put(52,25){\vector(1,0){21}}
\put(57,19){$s_3$}
\put(48,-12){$7$}
\put(50,0){\circle*{5}}
\put(50,23){\vector(0,-1){21}}
\put(41,10){$s_5$}
\put(73,30){$6$}
\put(75,25){\circle*{5}}
\put(77,25){\vector(1,0){21}}
\put(82,19){$s_5$}
\put(98,30){$9$}
\put(100,25){\circle*{5}}
\put(102,25){\vector(1,0){21}}
\put(107,19){$s_4$}
\put(120,30){$10$}
\put(125,25){\circle*{5}}
\end{picture}
\vspace*{0.7em}
\end{center}

We find the standard basis $\mathfrak B$ with associated matrix $B$
as shown below. (It is not always the case that the entries of $B$
are only monomials.) Hence we have 
$R=\operatorname{diag}[v^{d_1},\ldots,v^{d_{10}}]$, where 
$[d_1,\ldots,d_{10}]=[0,1,2,2,2,3,3,3,4,5]=[l(w_1),\ldots,l(w_{10})]$,
and $C$ is the identity matrix. Thus we get the matrix $\widetilde B$, 
and from that $\widehat b=1$ and the matrix $\widehat B$ as also shown below.
Note that the entries of $\widehat B$ are not necessarily palindromic or 
skew-palindromic, and that the maximum degree of the non-zero entries of 
$B$, $\widetilde B$ and $\widehat B$ equals $8$, $3$ and $5$, respectively:

{\small
$$ B= \begin{bmatrix}
.&.&.&.&.&.&.&.&.&1 \\
.&.&.&.&v&.&.&.&.&v^2 \\
.&.&.&.&v^3&.&.&.&v^2&. \\
.&.&.&.&v^3&.&.&v^2&.&. \\
.&.&.&v^2&v^3&.&.&.&.&. \\
.&.&.&.&v^5&.&v^3&v^4&v^4&v^4 \\
.&.&v^3&v^4&v^5&.&.&.&v^4&v^4 \\
.&v^3&.&v^4&v^5&.&.&v^4&.&v^4 \\
.&v^5&v^5&v^6&v^7&v^4&v^5&v^6&v^6&v^6 \\
v^5&v^7&v^7&v^6&.&v^6&v^7&v^6&v^6&v^8 \\
\end{bmatrix}
\quad 
\widetilde B= \begin{bmatrix} 
.&.&.&.&.&.&.&.&.&1 \\
.&.&.&.&1&.&.&.&.&v \\
.&.&.&.&v&.&.&.&1&. \\
.&.&.&.&v&.&.&1&.&. \\
.&.&.&1&v&.&.&.&.&. \\
.&.&.&.&v^2&.&1&v&v&v \\
.&.&1&v&v^2&.&.&.&v&v \\
.&1&.&v&v^2&.&.&v&.&v \\
.&v&v&v^2&v^3&1&v&v^2&v^2&v^2 \\
1&v^2&v^2&v&.&v&v^2&v&v&v^3 \\
\end{bmatrix} $$
$$ \widehat B= \begin{bmatrix}
2v^5-3v^3&-2v^4+3v^2&v^3-v&v^3-v&v^3-v&.&.&.&-v&1 \\
-v^3-v&v^2&.&-v&-v&.&.&1&.&. \\
-v^3-v&v^2&-v&.&-v&.&1&.&.&. \\
v^2&-v&.&.&1&.&.&.&.&. \\
-v&1&.&.&.&.&.&.&.&. \\
v^4+2v^2&-v^3&v^2&v^2&v^2&-v&-v&-v&1&. \\
-v^3-v&v^2&-v&-v&.&1&.&.&.&. \\
v^2&-v&.&1&.&.&.&.&.&. \\
v^2&-v&1&.&.&.&.&.&.&. \\
1&.&.&.&.&.&.&.&.&. \\
\end{bmatrix} $$ }


Similarly, we find the standard basis $\mathfrak B'$ 
with associated matrix $B'$. As it turns out we indeed have $R'=R$, 
and $C'$ is the identity matrix. This yields the matrix $\widetilde B'$
as shown below. Note that the entries of $\widetilde B'$ are not necessarily
palindromic or skew-palindromic, and that the maximum degree of the
non-zero entries of $\widetilde B'$ is $9$:

{\small 
$$ \widetilde B'= \left[\begin{array}{ccccc}
2v^3&v^5+2v^3+v&v^5+2v^3+v&-v^4-v^2&v^5+2v^3+v \\
-2v^2&v^6-v^2&v^6-v^2&v^3+v&-v^4-2v^2-1 \\
-v^5+v&-2v^5&-v^5+v&v^6+v^4&-v^7-2v^5-v^3 \\
-v^5+v&-v^5+v&-2v^5&v^6+v^4&-v^7-2v^5-v^3 \\
-v^5+v&-v^5+v&-v^5+v&-v^2-1&-v^7-2v^5-v^3 \\
2v^4&2v^4&2v^4&v^7-v^5&-v^8+v^4 \\
2v^4&2v^4&v^4-1&-v^5-v^3&-v^8+v^4 \\
2v^4&v^4-1&2v^4&-v^5-v^3&-v^8+v^4 \\
v^7+v^5-v^3+v&-2v^3&-2v^3&-v^6+v^4&-v^9+v^7-v^5-v^3 \\
-v^6-v^4+v^2-1&-2v^6&-2v^6&v^5-v^3&v^8-v^6+v^4+v^2 \\
\end{array}\right. $$
$$ \left.\begin{array}{ccccc}
-v^4-v^2&v^5+2v^3+v&-v^4-v^2&-v^4-v^2&-v^6-2v^4-2v^2-1 \\
-v^5+v^3&v^6-v^2&v^3+v&v^3+v&-v^7-v^5 \\
v^4-v^2&-v^5+v&v^6+v^4&-v^2-1&v^6+v^4 \\
v^4-v^2&-v^5+v&-v^2-1&v^6+v^4&v^6+v^4 \\
v^4-v^2&-2v^5&v^6+v^4&v^6+v^4&v^6+v^4 \\
-v^3+v&v^4-1&-v^5-v^3&-v^5-v^3&-v^5-v^3 \\
-v^3+v&2v^4&v^7-v^5&-v^5-v^3&-v^5-v^3 \\
-v^3+v&2v^4&-v^5-v^3&v^7-v^5&-v^5-v^3 \\
v^2-1&-2v^3&-v^6+v^4&-v^6+v^4&v^4+v^2 \\
v^7+v^5&-2v^6&v^5-v^3&v^5-v^3&-v^9+v^7 \\
\end{array}\right] $$ }

From this we get $Q=\widehat B\cdot\widetilde B'$. As it turns out
we already have $\gcd(Q)=1$, thus we may let $P=-Q$ be as shown below.
Indeed, independent verification shows that $P$ is a primitive
Gram matrix as desired, coinciding with the one already given in 
\cite[Example 4.9]{gemu}. Note that indeed $P$ is a completely dense 
matrix, all of whose entries are $6$-palindromic, where the maximum 
degree occurring is $6$, and that in accordance with Table \ref{Mmaxd0}
the largest coefficient occurring has absolute value $3$, and that 
the specialization $v\mapsto 0$ yields the identity matrix:

{\small
$$ \left[\begin{array}{ccccc}
v^6+3v^4+3v^2+1&2v^4+2v^2&2v^4+2v^2&-v^5-2v^3-v&2v^4+2v^2 \\
2v^4+2v^2&v^6+3v^4+3v^2+1&2v^4+2v^2&-v^5-2v^3-v&2v^4+2v^2 \\
2v^4+2v^2&2v^4+2v^2&v^6+3v^4+3v^2+1&-v^5-2v^3-v&2v^4+2v^2 \\
-v^5-2v^3-v&-v^5-2v^3-v&-v^5-2v^3-v&v^6+2v^4+2v^2+1&-v^5-2v^3-v \\
2v^4+2v^2&2v^4+2v^2&2v^4+2v^2&-v^5-2v^3-v&v^6+3v^4+3v^2+1 \\
-v^5-2v^3-v&-v^5-2v^3-v&-v^5-2v^3-v&v^4+v^2&-2v^3 \\
2v^4+2v^2&2v^4+2v^2&2v^4+2v^2&-2v^3&2v^4+2v^2 \\
-v^5-2v^3-v&-v^5-2v^3-v&-2v^3&v^4+v^2&-v^5-2v^3-v \\
-v^5-2v^3-v&-2v^3&-v^5-2v^3-v&v^4+v^2&-v^5-2v^3-v \\
-2v^3&-v^5-2v^3-v&-v^5-2v^3-v&v^4+v^2&-v^5-2v^3-v \\
\end{array}\right. $$ 
$$ \left.\begin{array}{ccccc}
-v^5-2v^3-v&2v^4+2v^2&-v^5-2v^3-v&-v^5-2v^3-v&-2v^3 \\
-v^5-2v^3-v&2v^4+2v^2&-v^5-2v^3-v&-2v^3&-v^5-2v^3-v \\ 
-v^5-2v^3-v&2v^4+2v^2&-2v^3&-v^5-2v^3-v&-v^5-2v^3-v \\
v^4+v^2&-2v^3&v^4+v^2&v^4+v^2&v^4+v^2 \\
-2v^3&2v^4+2v^2&-v^5-2v^3-v&-v^5-2v^3-v&-v^5-2v^3-v \\
v^6+2v^4+2v^2+1&-v^5-2v^3-v&v^4+v^2&v^4+v^2&v^4+v^2 \\
-v^5-2v^3-v&v^6+3v^4+3v^2+1&-v^5-2v^3-v&-v^5-2v^3-v&-v^5-2v^3-v \\
v^4+v^2&-v^5-2v^3-v&v^6+2v^4+2v^2+1&v^4+v^2&v^4+v^2 \\
v^4+v^2&-v^5-2v^3-v&v^4+v^2&v^6+2v^4+2v^2+1&v^4+v^2 \\
v^4+v^2&-v^5-2v^3-v&v^4+v^2&v^4+v^2&v^6+2v^4+2v^2+1 \\
\end{array}\right] $$ }
\end{abs}



\begin{thebibliography}{99.}\label{biblio}


\bibitem{cohen} Cohen, H.:
A course in computational algebraic number theory.
Graduate Texts in Mathematics {\bf 138}. Springer-Verlag, 1993.

\bibitem{davguywan}
Davenport, J., Guy, M., Wang, P.:
$P$-adic reconstruction of rational numbers.
SIGSAM Bulletin {\bf 16} (2) (1982), 2--33.

\bibitem{dixon} Dixon, J.:
Exact solution of linear equations using $p$-adic expansions.
Numer. Math. {\bf 40} (1982), 137--141.

\bibitem{GAP} The {\sf GAP} Group:
{\sf GAP} --- Groups, Algorithms, Programming ---
A System for Computational Discrete Algebra.
Version 4.8.5 (2016). \url{http://www.gap-system.org}.

\bibitem{vzG} von zur Gathen, J., Gerhard, J.:
Modern computer algebra. Third edition.
Cambridge University Press, 2013.

\bibitem{gari}
Garibaldi, S.: $E_8$, the most exceptional group. Bull. Amer. Math. Soc.
{\bf 53} (2016), 643--671.


\bibitem{myedin}
Geck, M.: Leading coefficients and cellular bases of Hecke algebras.
Proc. Edinburgh Math. Soc. {\bf 52} (2009), 653--677.

\bibitem{geha}
Geck, M., Halls, A,: On the Kazhdan--Lusztig cells in type $E_8$. 
Math. Comp. {\bf 84} (2015), 3029--3049.

\bibitem{chevie}
Geck, M., Hi{\ss}, G., L\"ubeck, F., Malle, G., Pfeiffer, G.:
{\sc Chevie}-A system for computing and processing generic character tables
for finite groups of Lie type, Weyl groups and Hecke algebras. Appl.
Algebra Engrg. Comm. Comput. {\bf 7} (1996), 175--210.

\bibitem{geja} 
Geck, M., Jacon, N.: Representations of Hecke algebras at roots of unity. 
Algebra and Applications {\bf 15}, Springer-Verlag, 2011. 

\bibitem{gemu} Geck, M., M\"uller, J.: 
James' conjecture for Hecke algebras of exceptional type I.
J. Algebra {\bf 321} (2009), 3274--3298.

\bibitem{gepf}
Geck, M., Pfeiffer, G.: Characters of finite Coxeter groups and
Iwahori--Hecke algebras. London Math. Soc. Monographs, New Series {\bf 21}.
Oxford University Press, 2000.

\bibitem{GMP} The {\sf GMP} Development Team:
{\sf GMP} --- The {\sf GNU} Multiple Precision Arithmetic Library.
Version 6.1.1 (2016). \url{http://www.gmp-lib.org}.

\bibitem{hw} Hardy, G., Wright, E.:
An introduction to the theory of numbers. Sixth edition.
Oxford University Press, 2008.

\bibitem{How}
Howlett, R.B.: $W$-graphs for the irreducible representations
of the Hecke algebra of type $E_8$. Private communication
with J.~Michel (December 2003).

\bibitem{HowYin}
Howlett, R.B., Yin, Y.: Computational construction of irreducible $W$-graphs
for types $E_6$ and $E_7$. J. Algebra {\bf 321} (2009), 2055--2067.

\bibitem{KaLu}
Kazhdan, D.A., Lusztig, G.: Representations of Coxeter groups and
Hecke algebras. Invent. Math. {\bf 53} (1979), 165--184.


\bibitem{Lusztig03}
Lusztig, G.: Hecke algebras with unequal parameters. CRM Monographs
Ser. {\bf 18}. Amer. Math. Soc., 2003.
Enlarged and updated version at \url{arXiv:0208154v2}.

\bibitem{shaw}
Lusztig, G.: Algebraic and geometric methods in representation theory
(September 2014). 
\url{arxiv:1409.8003}.


\bibitem{mllt}
Maier, P., Livesey, D., Loidl, H.-W., Trinder, P.: High-performance 
computer algebra: A Hecke algebra case study. {\it In:} 
Silva, F., Dutra, I., Santos Costa, V. (eds.):
Euro-Par 2014 Parallel Processing: 20th International Conference,
Porto (August 2014).
Lecture Notes in Computer Science {\bf 8632} (2014), 415--426.

\bibitem{jmich}
Michel, J.: The development version of the CHEVIE package of GAP3.
J. Algebra {\bf 435} (2015), 308--336; see also
\url{http://www.math.rwth-aachen.de/~CHEVIE/}.

\bibitem{mon}
Monagan, M.: Maximal quotient rational reconstruction: 
An almost optimal algorithm for rational reconstruction.
{\it In:} Proceedings of ISSAC '04, 243--249. ACM Press, 2004.




\bibitem{Naruse0}
Naruse, H.: $W$-graphs for the irreducible representations of the 
Iwahori--Hecke algebras of type $F_4$ and $E_6$. Private communication
with M.~Geck (January and July, 1998).

\bibitem{parker1}
Parker, R.A.: The computer calculation of modular characters (the Meat-Axe). 
{\it In:} Atkinson, M.D. (ed): Computational Group Theory, Durham (1982),
267--274. Academic Press, 1984.

\bibitem{parker2}
Parker, R.A.: An integral meataxe. {\it In:} The Atlas of Finite Groups: Ten
Years On, Birmingham (1995), 215--228. London Mathematical Society 
Lecture Note Series {\bf 249}. Cambridge University Press, 1998.



\bibitem{VoE8}
Vogan, D.A.: The character table of $E_8$. Notices of the Amer.
Math. Soc. {\bf 9} (2007), 1022--1034; see also
\url{http://atlas.math.umd.edu/AIM_E8/technicaldetails.html}.

 
\end{thebibliography}
\end{document}